\documentclass[reqno, 12pt]{amsart} 
\usepackage{array} 
\usepackage{amsmath}

\usepackage{amsfonts}
\usepackage{amssymb}
\usepackage{mathrsfs}
\usepackage{enumerate}
\usepackage{amsthm}
\usepackage{verbatim}
\usepackage{amsmath, amscd}
\usepackage[usenames,dvipsnames]{color}
\usepackage{enumitem}
\usepackage{bm}
\usepackage{xy}
\xyoption{all}
\usepackage{stmaryrd}
\usepackage{color}
\usepackage{marvosym}
\usepackage{amsbsy}
\usepackage{tikz, tikz-cd}
\usepackage{marginnote}
\usetikzlibrary{matrix,arrows,backgrounds}\usepackage[backgroundcolor=white]{todonotes}
\usepackage{hyperref}
\usepackage{thmtools}
\usepackage[capitalize]{cleveref}

\newtheorem{theorem}{Theorem}[section]

\newtheorem{lemma}[theorem]{Lemma}
\crefalias{lemma}{lemma}

\newtheorem{proposition}[theorem]{Proposition}
\crefalias{proposition}{proposition}

\newtheorem{corollary}[theorem]{Corollary}
\crefalias{corollary}{corollary}
\theoremstyle{definition}
\newtheorem{definition}[theorem]{Definition}
\newtheorem{remark}[theorem]{Remark}
\newtheorem{example}[theorem]{Example}
\newtheorem{conjecture}[theorem]{Conjecture}

\newcommand{\op}[1]{\operatorname{#1}}

\newcommand{\leftexp}[2]{{\vphantom{#2}}^{#1}{#2}}

\newcommand{\newterm}{\textsf}

\newcommand{\window}{\leftexp{=}{\kern-0.23em\operatorname{W}}^{\kern-0.21em =}}

\newcommand{\E}{\mathcal{E}}
\newcommand{\End}{\mathcal{E}\!\mathrm{nd}}

\newcommand{\tyler}[1]{{\color{blue} \sf TK: [#1]}}

\def\id{\op{{id}}}

\def\Z{\op{\mathbb{Z}}}
\def\C{\op{\mathbb{C}}}

\def\F{\op{\mathcal{F}}}

\def\O{\op{\mathcal{O}}}

\def\coker{\op{coker}}

\def\End{\op{End}}

\title[Chern character for hypersurfaces]{The Chern Character of a coherent sheaf \\ on a smooth projective hypersurface}

\author[Favero]{David Favero}
\address{
	\begin{tabular}{l}
		David Favero \\
		\hspace{.1in} University of Minnesota, School of Mathematics \\
		\hspace{.1in} 206 Church Street, Minneapolis MN 55455, USA \\
		\hspace{.1in} Email: {\bf favero@umn.edu} \\
	\end{tabular}
}

\author[Kelly]{Tyler L. Kelly}
\address{
  \begin{tabular}{l}
   Tyler L. Kelly \\
   \hspace{.1in} Queen Mary University of London, School of Mathematical Sciences \\
   \hspace{.1in} Mile End Road, E1 4UJ, London,  United Kingdom \\
   \hspace{.1in} Email: {\bf t.kelly.1@bham.ac.uk} \\
  \end{tabular}
}

\numberwithin{equation}{section}
\begin{document}

\begin{abstract}
    Given a coherent sheaf on a smooth projective hypersurface $X$, we prove an explicit formula for its Chern character as a \v{C}ech cocycle in terms of the free resolution of the associated module and calculate its image in the Jacobian ring under the Griffiths residue map.  The formula is a geometric analogue of the Kapustin-Li formula for Landau-Ginzburg models, but proven directly using Hodge-theoretic techniques.  This yields an effective method to compute the primitive part of the Chern character of any coherent sheaf using commutative algebra. We finish by proving the Hodge conjecture for the degree 33 Fermat fourfold.
\end{abstract}
\maketitle

\section{Introduction}
Let $S = k[x_0, \dots, x_{n+1}]$ and let $Q \in S$ be a non-zero homogeneous polynomial of degree $m$. Consider the quotient ring $R = S/(Q)$. If $I = \langle g_1, \dots, g_s \rangle$ is a regular sequence, then, going back to Tate \cite{Tate}, the Koszul resolution of $S/I$ can be transformed into a resolution of $R/I$ by adding an extra homological generator in degree 2. The tail end of this resolution is called a matrix factorization; it is a 2-periodic complex with differentials $A, B$ which, when lifted back to $S$, satisfy $AB = BA = Q \cdot \id$.

In fact, the Auslander--Buchsbaum formula \cite{AuslanderBuchsbaum} dictates that the maximal depth of any $R$-module is $n+2$, the Krull dimension of $R$. A module which attains this depth is called a maximal Cohen--Macaulay (MCM) module. Given an MCM $R$-module $M$, the Auslander--Buchsbaum formula also tells us that its projective dimension is $1$ as an $S$-module. Hence, it has a 2-term minimal resolution
\[
0 \to E_1 \to E_0 \to M \to 0.
\]
If $M$ has no free summands, Eisenbud \cite{Eisenbud} showed that the minimal resolution of $M$ as an $R$-module is 2-periodic of the form
\[
\cdots \xrightarrow{A} E_0(-m) \xrightarrow{B} E_1 \xrightarrow{A} E_0 \to M \to 0.
\]

Furthermore, since the depth of a module increases with each syzygy until it reaches its maximum, given any finitely generated $R$-module $N$, the $(n+1)$-th syzygy $M$ is an MCM module. Hence, any finitely generated $R$-module $N$ has a minimal resolution which eventually becomes 2-periodic:
\[
\cdots \xrightarrow{A} E_0(-m) \xrightarrow{B} E_1 \xrightarrow{A} E_0 \to V_t \to \dots \to V_0 \to N \to 0
\]

To translate to geometry, let $X := Z(Q) \subseteq \mathbb{P}^{n+1}$ be a smooth projective hypersurface of dimension $n$. Consider a coherent sheaf $\mathcal{G} = \widetilde{N}$. Then $\mathcal{G}$ has a resolution 
\begin{equation}\label{eq: intro resolution of G}
\cdots \xrightarrow{A} \mathcal{E}_0(-m) \xrightarrow{B} \mathcal{E}_1 \xrightarrow{A} \mathcal{E}_0 \to \mathcal{V}_t \to \cdots \to \mathcal{V}_0 \to \mathcal{G} \to 0,
\end{equation}
where the $\mathcal{E}_i, \mathcal{V}_j$ are direct sums of line bundles of the form $\mathcal{O}_X(d)$. Furthermore, since $X$ is smooth, a theorem of Grothendieck \cite{GrothendieckSGA2} tells us that the top syzygy is a vector bundle. That is, we have a locally-free resolution
\begin{equation}
\label{eq:resolution_of_a_sheaf}
0 \to \coker A \to \mathcal{V}_t \to \dots \to \mathcal{V}_0 \to \mathcal{G} \to 0
\end{equation}
and $\coker A$ is precisely the sheaf corresponding to the MCM module which is the top syzygy of $N$.  If we take $N$ to be the unique saturated module which sheafifies to $\mathcal G$, then $\coker A$ is unique as well. In geometry, vector bundles of the form $\coker A$ are called \newterm{arithmetically Cohen--Macaulay (ACM) sheaves}. 

Since the Chern character of a coherent sheaf descends to $K$-theory, the exact sequence \eqref{eq:resolution_of_a_sheaf} tells us that
\begin{equation}\label{intro eq: chern of G}
ch(\mathcal{G}) = ch(\mathcal V_0) + \dots + (-1)^{t} ch(\mathcal V_t) + (-1)^{t+1} ch(\coker A).  
\end{equation}
That is, $ch(\mathcal G)$ agrees with the Chern character of an ACM sheaf up to sign and hyperplane classes.  Hence, to study the image of the Chern character map, it suffices to study Chern characters of ACM sheaves. 

Our first result is a formula for the Chern character of an ACM sheaf $ch( \mathcal F)$ as a \v{C}ech cocycle based entirely on the differentials $A,B$ in the 2-periodic resolution of the corresponding MCM module.
Our ACM sheaf $\coker A$ has a resolution given by $$\dots \xrightarrow{A} \E_0(-m) \xrightarrow{B} \E_1 \xrightarrow{A} \E_0 \longrightarrow \coker A.$$
 Write $Q_i := \partial_iQ$ for the $i$th partial derivative. Similarly, $\partial_i A$ (resp. $\partial_i B$) is the matrix obtained by taking the partial derivative with respect to $x_i$ on all entries of $A$ (resp. $B$).
  Define an open cover $\mathcal{U}$ on X given by 
$$
 \mathcal{U} := \{ U_{i,t} = D(x_tQ_i) \ | \ i,t \in \{0,\dots, n+1\}\}, 
 $$
 where  $D(x_tQ_i):= \{(x_i) \ | \ x_tQ_i\ne0\}$.
 Take the cohomology classes
$$
\Theta :=  \left(\left(\Lambda_{\mathbf{d}}\frac{\partial_{i}A}{x_{t}Q_i}dx_{t}  - \Lambda_{\mathbf{d}}\frac{\partial_{j}A}{x_uQ_j} dx_{u} \right)B \right)_{U_{i,t} \cap U_{j,u}} \in H^1(X, \End(\mathcal{\E}_0) \otimes \Omega_X^1)
$$
and 
$$
\Xi := \left(\frac{dA \partial_{i}B\partial_{j}AB}{Q_i Q_j}\right)_{U_{i,t} \cap U_{j,u}} \in H^1(X, \End(\mathcal{\E}_0) \otimes \Omega_X^1),
$$
viewed as \v{C}ech cocycles with respect to the cover $\mathcal U$.

\begin{theorem}[= \Cref{thm: chern character}]\label{intro thm: chern character}
The $k$th Chern character $\op{ch}_k(\coker A)$ of $\coker A$ is the \v{C}ech cocycle given by
$$
\op{ch}_k(\coker A) = \frac{1}{k!} \op{tr}((\Theta - \Xi)^k).
$$
\end{theorem}
This theorem is proven by finding local algebraic connections and computing the Atiyah class explicitly.

Since $\mathcal{V}_i = \oplus_{j=1}^{r_j} \O_X(d_{ij}H)$, $\op{ch}(\mathcal{V}_i) = \sum_{j=1}^{r_j}e^{d_{ij}H}$, which yields an explicit formula for $\op{ch}(\mathcal{G})$. The explicit \v{C}ech cocycle for the hyperplane class is $H = (\frac{dx_i}{x_i} -\frac{dx_j}{x_j})_{U_{i,t} \cap U_{j,u}}$ for all $t,u$.  Plugging this into~\eqref{intro eq: chern of G}, we get the following formula:
\begin{corollary}\label{cor: Cech chern G}
A coherent sheaf $\mathcal G$ on a smooth projective hypersurface $X$ has Chern character
$$
\op{ch}_k(\mathcal G) = \frac{1}{k!}\left(\sum_{i=0}^t (-1)^t\left(\sum_{j=1}^{r_j} d_{ij}H\right)^k +(-1)^{t+1}  \op{tr}\left((\Theta - \Xi)^k\right)\right).
$$
\end{corollary}

On a hypersurface, by the Lefschetz hyperplane theorem, the only cohomology not generated by the hyperplane class lies in $H^n(X, \C)$.  Define the primitive cohomology $H_{\op{prim}}^n(X,\C)$ of $X$ to be the subspace of $H^n(X,\C)$ that is orthogonal to the hyperplane class. There, we have the added benefit of the Griffiths Residue Theorem \cite{Gri} which identifies the primitive cohomology of $X$ with graded pieces of the Jacobian ring (see~\Cref{Griffiths Residue Theorem} for the precise statement)
$$
\op{Jac}(Q) := \mathbb{C}[x_0, \dots, x_{2k+1}]/\langle Q_0, \dots, Q_{n+1}\rangle.
$$
As the Chern character lives in $\oplus_{p}H^{p,p}(X, \C)$, it can only hit the primitive cohomology of $X$ when $X$ has even dimension. 
Hence, moving forward we restrict our attention to the case where $n = 2k$ and study the primitive target of the Chern character map as a graded piece of the Jacobian ring under Griffiths' isomorphism
$$
H^{k,k}_{\op{prim}}(X) \cong \op{Jac}(Q)_{(k+1)\deg Q -(2k+2)}.
$$

The main result of this paper is an explicit formula for the primitive component of the Chern character of any coherent sheaf expressed as a homogeneous element of the Jacobian ring. 
\begin{theorem}[=\Cref{thm: KL}]\label{intro thm: KL} 
Let $\mathcal{G}$ be a coherent sheaf on a smooth $2k$-dimensional projective hypersurface $X$. Consider the resolution given by~\eqref{eq: intro resolution of G}.
The projection $\op{ch}_k^{\op{prim}}(\mathcal G)$ of the $k$th Chern character to the primitive cohomology of $X$  viewed as an element of the Jacobian ring is given by the formula
$$
\op{ch}_k^{\op{prim}}(\mathcal G) = \frac{(-1)^{{k+1 \choose 2 }}k!}{m}\op{tr}( \partial_{0}A\partial_1B \cdots \partial_{2k}A\partial_{2k+1} B -  \partial_{0}B\partial_1A \cdots \partial_{2k}B\partial_{2k+1} A).
$$
 If the resolution of $\mathcal{G}$ by direct sums of line bundles of the form $\O_X(d)$ is bounded, then $\E_0$ and $\E_i$ are the trivial sheaf and $\op{ch}_k^{\op{prim}}(\mathcal G)=0$.
\end{theorem}
The proof of the above theorem leverages the explicit \v{C}ech representative of \Cref{cor: Cech chern G} together with work of Carlson and Griffiths \cite{CG} which provides the explicit \v{C}ech cocycles corresponding to polynomials in the Jacobian ring.  Specifically, we trace both the Carlson-Griffiths expression and the expression in \Cref{intro thm: chern character} through a sequence of connecting homomorphisms
\[
\delta_{\ell}: H^{\ell}(X,\Omega_X^{n-\ell}) \longrightarrow H^{\ell+1}(X, \Omega_{X}^{n-\ell-1}(-\deg Q)),
\]
which identifies both \v{C}ech cocycles with polynomials.
We note that the formula found in~\Cref{intro thm: KL} is satisfyingly recognizable as a Kapustin-Li formula, which computes Chern characters for categories of matrix factorizations \cite{KapustinLi, DyckerhoffMurfet, PV12}.

A fun consequence is a rephrasing of the Hodge Conjecture for smooth projective hypersurfaces in terms of matrices.
\begin{conjecture}\label{conj: MF Hodge}
    Suppose $X = Z(Q)\subseteq \mathbb{P}^{2k+1}$ is a $2k$-dimensional projective hypersurface. Suppose $f \in \op{Jac}(Q)$. The polynomial $f$ represents a Hodge class in $H^{k,k}_{\op{prim}}(X)$ if and only if there exists square matrices $A_1, B_1, \dots, A_N, B_N$ with homogeneous polynomial entries in $\mathbb C[x_0,\dots, x_{2k+1}]$ so that 
    \begin{enumerate}
        \item $A_i B_i = B_i A_i =  Q\cdot \id$, and
        \item There exists constants $c_i \in \mathbb{Q}$ so that
        \begin{equation}\label{eqn:MF hodge intro}
        f = \sum_{i=1}^N c_i \left( \op{tr}( \partial_{0}A_i\partial_1B_i \cdots \partial_{2k}A_i\partial_{2k+1} B_i -  \partial_{0}B_i\partial_1A_i \cdots \partial_{2k}B_i\partial_{2k+1} A_i)\right)\in \op{Jac}(Q).
        \end{equation}
    \end{enumerate}
\end{conjecture}
Note that \Cref{intro thm: KL} implies the following:
\begin{corollary}\label{intro cor: HC for Proj Hypers}
The Hodge conjecture is true for $X$ if and only if \Cref{conj: MF Hodge} is.
\end{corollary}

If we specialize even further to a Fermat hypersurface, the complexified Hodge classes have been classified by Katz \cite{Katz} and Ogus \cite{Ogus} (see \cite[Theorem I]{ShiodaHodge}). In light of this classification, we define
$$
\mathcal{B}_m^{2k} = \left\{ \mathbf{d} =(d_0, \dots, d_{2k+1}) \ \middle|
 \  d_i \in \{0, \dots, m-2\}, \text{ and } \sum_{i=0}^{n+1} \left\langle\tfrac{d_i+1}{m}\right\rangle = k+1 \text{ for all $(\Z / m\Z)^\times$}\right\}.
$$

Using \Cref{intro thm: KL}, \cite[Theorem I]{ShiodaHodge}, \Cref{prop: g action on MF}, and \Cref{prop: isolate algebraics fermat}, the Hodge conjecture for the Fermat $2k$-fold of degree $m$  specializes to the following simple statement.
\begin{conjecture}\label{conj: Fermat MF Hodge}
Let $Q := \sum_{i=0}^{2k+1}x_i^m$.
There exists square matrices $A,B$ with homogeneous polynomial entries in $\mathbb C[x_0,\dots, x_{2k+1}]$ such that $A B = B A =  Q\cdot \id$ and the polynomial 
        $$
         \op{tr}( \partial_{0}A\partial_1B \cdots \partial_{2k}A\partial_{2k+1} B -  \partial_{0}B\partial_1A \cdots \partial_{2k}B\partial_{2k+1} A) = \sum_{\mathbf{d} \in \mathcal{B}_m^{2k}} c_{\mathbf{d}} x_0^{d_0}\cdots x_{2k+1}^{d_{2k+1}} \in \op{Jac}(Q)
        $$ 
 where $c_{\mathbf{d}} \neq 0$ for all $\mathbf{d}$.
\end{conjecture}

This conjecture is known when $m$ is a prime power and  for small $m,n$ \cite{ShiodaHodge, Ran, Aoki, daSilva} (see \Cref{thm: known Fermat Hodge}).  
All these cases find the relevant Hodge classes by providing explicit $k$-dimensional complete intersections in $\mathbb{P}^{2k+1}$ lying inside the hypersurface $X=Z(Q)$.
As the structure sheaf of these complete intersections have Tate's aforementioned explicit  resolution  \cite{Tate}. By applying \Cref{thm: KL} to these special resolutions, we recover the primitive part of Villaflor's formula.
\begin{theorem}[{=\Cref{thm: koszul det chprim}, \cite[Theorem 1.1]{loyola2024toric}}]\label{intro thm: Koszul}
Suppose that $Q = \sum_{i=0}^k a_ib_i$ and that $Z = V(a_0, ..., a_{k})$ is a complete intersection in $\mathbb P^{2k+1}$.  Consider the matrix
\[
M_Z := \begin{bmatrix}
\partial_0 a_0 & \cdots& \partial_0 a_k & \partial_0 b_0 & \cdots & \partial_0 b_k  \\
\vdots & &\vdots & \vdots & & \vdots \\
\partial_{2k+1} a_0 & \cdots & \partial_{2k+1} a_k &  \partial_{2k+1} b_0 &  \cdots & \partial_{2k+1} b_k  \\
\end{bmatrix}
\]
Then
\[
\op{ch}^{\op{prim}}_k(\mathcal O_Z) = (-1)^{k+1}\op{det}(M_Z).
\]
\end{theorem}
This gives a computable framework to easily test candidates for new algebraic cycles. For example, we are able to answer a question of da Silva \cite[Question 1]{daSilva} in the negative (see \Cref{ex:daSilva}).

There, da Silva aimed to create an algebraic cycle on the degree 33 Fermat fourfold, a case where the Hodge conjecture was open.   Nevertheless, in \Cref{sec: Hodge for deg 33} we prove it.
\begin{theorem}\label{intro thm HC for 33 4}
    The Hodge conjecture is true for the degree 33 Fermat fourfold.
\end{theorem}

The proof of \Cref{intro thm HC for 33 4}, reduces to finding some specific new algebraic cycles.  To find them, consider the cubic fourfold $Y$ with  defining equation
$$
x_0^2 x_1 + x_1^2 x_2 + x_2^2 x_3 + x_3^2 x_4 + x_4^2 x_0 + x_3^3.
$$
By work of Billi, Grossi, and Marquand, the algebraic lattice of $Y$ has full rank and is generated by rational normal scrolls \cite{BGM}. The pullbacks of these classes via a special rational map known as a Shioda map (see~\eqref{eq: shioda map 1}) to the degree 33 Fermat fourfold yield the required cycles. These rational normal scrolls are not complete intersections in $\mathbb P^5$. We expect our new algebraic cycles cannot be obtained from complete intersections,  like the other known Hodge classes for Fermat hypersurfaces (see~\Cref{rmk: no Koszul}).

\subsection{Acknowledgments}
The authors would like to thank Christine Berkesch, Andrea Brini, Michael Brown, Ron Donagi, Mark Gross, Lisa Marquand, Richard Thomas,  Tony Pantev, and Mark Shoemaker for discussions relating to this work. The first author was supported by the NSF under DMS award numbers 2302262 and 2412039. The second author was supported by the UKRI Future Leaders Fellowship MR/T01783X/1, its renewal MR/Y033841/1, and the EPSRC Mathematical Sciences Small Grant EP/Y033574/1. He also acknowledges the hospitality of the Sydney Mathematical Research Institute and the University of Minnesota where portions of this research were performed. 

\subsection{AI Declaration}
Google Gemini and Claude were used primarily as a glorified search engine, confirming intuition, verifying our memory of well-known results, and identifying references. Mathematically, Gemini identified the local idempotent in ~\Cref{lem: splitting} and suggested considering the short exact sequence ~\eqref{eqn:Poincare}, both without proof. All logical arguments, proofs and writing are the authors' own. The authors take full responsibility for the correctness of all results.

\section{A \v{C}ech formula for the Chern character of a coherent sheaf on a hypersurface}\label{sec: Cech}

Take a smooth degree $m$ hypersurface $X = Z(Q)\subseteq \mathbb{P}^{n+1}$.
\begin{definition}
A coherent sheaf $\mathcal{F}$ on $X$ is called an \emph{arithmetically Cohen--Macaulay (ACM) sheaf} if it is locally free and has no intermediate vanishing cohomology, i.e.,
\[
H^i(X, \mathcal{F}(k)) = 0 \quad \text{for all } 1 \le i \le n-1 \text{ and all } k \in \mathbb{Z}.
\]
Equivalently, its graded section module $\Gamma_*(\mathcal{F}) = \bigoplus_{k \in \mathbb{Z}} H^0(X, \mathcal{F}(k))$ is a maximal Cohen--Macaulay $R$-module over the coordinate ring $R$.
\end{definition}

Given an arithmetically Cohen-Macaulay sheaf $\mathcal{F}$ on $X$, Eisenbud \cite{Eisenbud} proves there exists a 2-periodic resolution
\begin{equation}\label{def: FF}
\cdots \stackrel{A}{\longrightarrow} \mathcal{E}_{0}(-2m) \stackrel{B}{\longrightarrow} \mathcal{E}_1(-m)\stackrel{A}{\longrightarrow} \mathcal{E}_0(-m) \stackrel{B}{\longrightarrow} \mathcal{E}_1 \stackrel{A}{\longrightarrow} \mathcal{E}_0 \stackrel{\pi}{\longrightarrow} \mathcal{F} \longrightarrow 0,\end{equation}
where 
$$
\mathcal{E}_0 = \bigoplus_{\ell=1}^r \O(d_\ell), \quad \mathcal{E}_{1} = \bigoplus_{\ell=1}^r \O(e_\ell), 
$$
 $A\circ B = Q \cdot \op{id}$ and $B \circ A(-d) = Q \cdot \op{id}$. Note that we have an isomorphism $\mathcal{F} \cong \coker A$. 

 Since $X$ is smooth, we can define an open cover $\mathcal{U}$ on $X$ to be
 \begin{equation}\label{the refined cover}
 \mathcal{U} := \{ U_{i,t} = D(x_tQ_i)   \ | \ i,t \in \{0,\dots, n+1\}\},
 \end{equation}
 where $D(x_tQ_i):= \{(x_i) \ | \ x_tQ_i\ne0\}$.

\begin{lemma} \label{lem: splitting}
On each $U_{i,t}$, there is a splitting 
$$
s_{i}: \mathcal{F}|_{U_{i,t}} \to \E_0|_{U_{i,t}}
$$
so that $\pi \circ s_{i} = \id$ and $$s_{i} \circ \pi = \frac{\partial_i A B}{Q_i}$$ is an idempotent.
\end{lemma}

\begin{proof}
Since $BA = Q$, the map $\frac{\partial_i A B}{Q_i}: \mathcal E_0 \to \mathcal E_0$ is trivial on the image of $A$ and hence descends to the cokernel $\mathcal F$.  The descended map is $s_i$.

To check that $\frac{\partial_i A B}{Q_i}$ is idempotent, simply compute
\begin{equation*}
    \frac{\partial_i A B}{Q_i}  \cdot  \frac{\partial_i A B}{Q_i} =     \frac{\partial_i A (B\partial_i A) B}    {Q_i^2}  = \frac{\partial_i A (Q_i - \partial_iBA) B}    {Q_i^2}  =  \frac{\partial_i A B}{Q_i}. \qedhere
\end{equation*}
\end{proof}
As $U_{i,t} \subseteq D(x_t)$, we denote the restriction of the standard trivialization of $\mathcal E_0$ on $D(x_t)$ to $U_{i,t}$ by $
 \psi_t: \mathcal E_0|_{U_{i,t}} \to \mathcal O_{U_{i,t}}^{\oplus n}$.
This allows us to define the standard local algebraic connections
$$
\nabla^{\E_0}_t := \psi_t^{-1}\circ d\circ \psi_t: \E_0|_{U_{i,t}} \to \E_0|_{U_{i,t}} \otimes \Omega^1_{U_{i,t}}.
$$

\noindent The local splitting of $\mathcal F$ from \Cref{lem: splitting}  induces a local algebraic connection on $\mathcal F$ as well.
\begin{proposition}\label{prop: connection}
On each affine open, we have a local algebraic connection 
\[
\nabla_{i,t} := (\pi\otimes \id)\circ \nabla^{\E_0}_t\circ s_{i}: \mathcal{F}|_{U_{i,t}} \to \Omega^1_{U_{i,t}} \otimes \mathcal{F}|_{U_{i,t}}.
\]
\end{proposition}

\begin{proof}
 This is a formal consequence of \Cref{lem: splitting} and the fact that $\nabla_t^{\mathcal E_0}$ is an algebraic connection.  Indeed, \begin{align*}
(\pi \otimes \id) \circ \nabla_t^{\mathcal E_0} \circ s_i (r \cdot a) & = (\pi \otimes \id) \circ \nabla_t^{\mathcal E_0} ( r \cdot  s_i (a)) \\ &= (\pi \otimes \id) (s_i (a) \otimes dr + r \cdot \nabla_t^{\mathcal E_0}(s_i(a))) \\& = a \otimes dr + r \cdot  (\pi \otimes \id) \circ \nabla_t^{\mathcal E_0} \circ s_i(a). \qedhere
 \end{align*}
\end{proof}
Given a local algebraic connection, the Atiyah class can be defined as follows.
\begin{definition} Given an open cover $ \{U_i \ | \ i \in \{1, \dots, N\}\}$ of an algebraic variety $X$ and vector bundle $\F$ equipped with local connections $\nabla_{U_i}: \F \to \F \otimes \Omega_{U_i}^1$, 
the \newterm{Atiyah class} $\op{at}(\mathcal{F})$ is the \v{C}ech cocycle given locally on $U_{ij} := U_i \cap U_j$ by
$$
\op{at}(\mathcal{F})_{ij} := \nabla _i -\nabla_j \in \Gamma(U_{ij}, \Omega_{U_{ij}}^1 \otimes \End(\mathcal{F})).
$$
\end{definition}

\begin{example}\label{ex: Atiyah of direct sum}
Let $\{ D(x_t) \ | \ t \in \{0, ..., n+1\}\}$ be the standard open cover of $\mathbb P^{n+1}$.  Let $\mathcal E_0 := \bigoplus_{\ell=1}^r \mathcal O(d_\ell)$.  Using the standard connection
\begin{align*}
\op{at}(\E_0)_{tu}(a_1, \dots, a_r) & =\nabla^{\E_0}_t(a_1,\dots, a_r) - \nabla^{\E_0}_u(a_1, \dots, a_r) \\
& = \psi_t^{-1}\circ d\circ \psi_t (a_1, \dots, a_r) - \psi_u^{-1}\circ d\circ \psi_u (a_1, \dots, a_r) \\
& = \psi_t^{-1}d(x_t^{d_1}a_1, \dots, x_t^{d_r}a_r) -\psi_u^{-1}d(x_u^{d_1}a_1, \dots, x_u^{d_r}a_r) \\
& = \psi_t^{-1}(d_1x_t^{d_1-1}a_1, \dots, d_rx_t^{d_r-1}a_r)dx_t+\psi_t^{-1}(x_t^{d_1}da_1, \dots, x_u^{d_r}da_r) \\&\quad -  (\psi_u^{-1}(d_1x_u^{d_1-1}a_1, \dots, d_rx_u^{d_r-1}a_r)dx_u+\psi_u^{-1}(x_u^{d_1}da_1, \dots, x_u^{d_r}da_r)) \\
&= \Lambda_{\mathbf{d}}\left(\frac{dx_t}{x_t} - \frac{dx_u}{x_u} \right)(a_1, \dots, a_r),
\end{align*}
where $\Lambda_{\mathbf{d}} :=\op{diag}(d_1, \dots, d_r)$ is the diagonal matrix with entries $(\Lambda_{\mathbf{d}})_{ij} = \delta_{ij}d_i$.
In other words,
\begin{equation} \label{eq: ambient cxn}
    \op{at}(\E_0)_{tu} = \Lambda_{\mathbf{d}}\left(\frac{dx_t}{x_t} - \frac{dx_u}{x_u} \right).
\end{equation}

\end{example}

Recall that our open cover has the form $\mathcal U = \{U_{i,t} \ | \ i,t \in \{0, \dots, n+1\}$. We use the notation $\op{at}(\mathcal{F})_{ij,tu}$ for the Atiyah class on the open set $U_{ij, tu}:= U_{i,t} \cap U_{j, u}$.  A reduced expression for this \v{C}ech cocycle is  given in the following proposition.

\begin{proposition}\label{prop: atiyah class}
    The Atiyah class $\op{at}(\mathcal{F}) \in H^1(X, \End(\mathcal F) \otimes \Omega_X^1)$ of $\mathcal{F}$ is given by the \v{C}ech cocycle 
    $$
\op{at}(\mathcal{F})_{ij,tu} = (\pi \otimes \id )\left( \Lambda_{\mathbf{d}} \frac{dx_t }{x_t} \right) s_i - (\pi\otimes \id)\left(\Lambda_{\mathbf{d}}\frac{d x_{u}}{x_{u}}\right) s_{j}  -    (\pi\otimes \id) \circ dA\left(\frac{\partial_{i}B}{Q_i} \right)s_{j}.
    $$

\end{proposition}

\begin{proof}
We write 
    \begin{align}
        \nabla_{i, t} - \nabla_{j,u} &= (\pi\otimes \id)\circ \nabla^{\E_0}_t\circ s_{i} - (\pi\otimes \id)\circ \nabla^{\E_0}_{u}\circ s_{j} \\
        &= (\pi\otimes \id)\circ (\nabla^{\E_0}_t-\nabla^{\E_0}_{u})\circ s_{i} + (\pi\otimes \id)\circ \nabla^{\E_0}_{u}\circ (s_{i}- s_{j}). \label{eq: two summands in atiyah}
    \end{align}
By \eqref{eq: ambient cxn}, the first summand in ~\eqref{eq: two summands in atiyah} is
\begin{equation}\begin{aligned}\label{atiyah 1}
    (\pi\otimes \id)\circ (\nabla^{\E_0}_t-\nabla^{\E_0}_{u})\circ s_{i} &= (\pi \otimes \id ) \left( \Lambda_{\mathbf{d}} \left(\frac{dx_t }{x_t} - \frac{d x_{u}}{x_{u}}\right)\right) s_i.
\end{aligned}\end{equation}

Now turn to the second summand in ~\eqref{eq: two summands in atiyah}.  On the open set $U_{i,t}\cap U_{j,u}$, we write any section $\pi g \in \F(U_{ij,tu})$, where $g \in \E_0(U_{ij, tu})$. We then compute
\begin{equation}\begin{aligned}\label{atiyah 2}
    (\pi\otimes \id)\circ \nabla^{\E_0}_{u}\circ (s_{i}- & s_{j})(\pi g) = (\pi\otimes \id)\circ (\psi_{u}^{-1}\otimes \id)\circ d\circ \psi_{u}\circ \left(\frac{\partial_{i}AB}{Q_i} -\frac{\partial_{j}AB}{Q_j} \right) (g) \\
   &= (\pi\otimes \id)\circ (\psi_{u}^{-1}\otimes \id)\circ d( \psi_{u})\circ \left(\frac{\partial_{i}AB}{Q_i} -\frac{\partial_{j}AB}{Q_j} \right) (g) \\
   &\qquad +  (\pi\otimes \id)\circ (\psi_{u}^{-1}\otimes \id)\circ ( \psi_{u}\otimes \id)\circ d\left(\frac{\partial_{i}AB}{Q_i} -\frac{\partial_{j}AB}{Q_j} \right) (g)  \\
   &\qquad + (\pi\otimes \id)\circ (\psi_{u}^{-1}\otimes \id)\circ  \left(\psi_{u}\circ \left(\frac{\partial_{i}AB}{Q_i} -\frac{\partial_{j}AB}{Q_j} \right)  \otimes \id \right)(dg) \\
   &= (\pi \otimes \id) \left( \Lambda_{\mathbf{d}}  \frac{d x_{u}}{x_{u}}\right)( s_i - s_{j})(\pi g)   +  (\pi\otimes \id)\circ d\left(\frac{\partial_{i}AB}{Q_i} -\frac{\partial_{j}AB}{Q_j} \right) (g)  \\
   &\qquad + (\pi\circ (s_i - s_{j})\circ \pi)\otimes \id)(dg) \\
     &= (\pi \otimes \id) \left( \Lambda_{\mathbf{d}}  \frac{d x_{u}}{x_{u}}\right)( s_i - s_{j})(\pi g)   +  (\pi\otimes \id)\circ d\left(\frac{\partial_{i}AB}{Q_i} -\frac{\partial_{j}AB}{Q_j} \right) (g).
\end{aligned} 
\end{equation}
Focusing on the latter summand, we compute 
\begin{equation}
\frac{\partial_iAB}{Q_i}-\frac{\partial_{j}AB}{Q_j}    =-\frac{A\partial_{i}B\partial_{j}AB}{Q_iQ_j}.
\end{equation}
Continuing, we obtain
\begin{equation}\begin{aligned}\label{atiyah 3}
(\pi\otimes \id)\circ d&\left(\frac{\partial_{i}AB}{Q_i} -\frac{\partial_{j}AB}{Q_j} \right) (g)  \\
	&= -   (\pi\otimes \id) \circ dA\left(\frac{\partial_{i}B\partial_{j}AB}{Q_iQ_j} \right)(g) - (\pi\otimes \id) \circ (A\otimes \id)d\left(\frac{\partial_{i}B\partial_{j}AB}{Q_iQ_j} \right)(g)\\
    &= - (\pi\otimes \id) \circ dA\left(\frac{\partial_{i}B\partial_{j}AB}{Q_iQ_j} \right)(g)\\
   &= -(\pi\otimes \id) \circ dA\left(\frac{\partial_{i}B}{Q_i} \right)s_{j} (\pi g)
\end{aligned}\end{equation}
where the third equality follows from $\pi\circ A = 0$ and the fourth from Lemma~\ref{lem: splitting}.
Lastly, we use ~\eqref{eq: two summands in atiyah}, ~\eqref{atiyah 1}, \eqref{atiyah 2}, and \eqref{atiyah 3} to conclude
\begin{align*}
\nabla _{i,t} -\nabla_{j,u}    &= (\pi \otimes \id ) \left( \Lambda_{\mathbf{d}} \left(\frac{dx_t }{x_t} - \frac{d x_{u}}{x_{u}}\right)\right) s_i  + (\pi \otimes \id) \left( \Lambda_{\mathbf{d}}  \frac{d x_{u}}{x_{u}}\right)( s_i - s_{j})\\
	&\qquad   -   (\pi\otimes \id) \circ dA\left(\frac{\partial_{i}B}{Q_i} \right)s_{j}  \\
	&=  (\pi \otimes \id )\left( \Lambda_{\mathbf{d}} \frac{dx_t }{x_t} \right) s_i - (\pi\otimes \id)\left(\Lambda_{\mathbf{d}}\frac{d x_{u}}{x_{u}}\right) s_{j}  -    (\pi\otimes \id) \circ dA\left(\frac{\partial_{i}B}{Q_i} \right)s_{j}. \qedhere
\end{align*}
\end{proof}

We define the Chern character as the trace of the exponential of the Atiyah class, following his classical work \cite{Atiyah}.

\begin{definition}The \newterm{Chern character} of a coherent sheaf $\mathcal{G}$ is  
$$
\op{ch}(\mathcal{G}) = \op{tr}(\op{exp}(\op{at}(\mathcal{G}))) \in \bigoplus_k H^k(X, \Omega^k_X).
$$
Write $\op{ch}(\mathcal{G})  = \sum_k \op{ch}_k(\mathcal{G})$, where $\op{ch}_k(\mathcal{G}) \in H^k(X, \Omega^k_X)$. In this case, we call $\op{ch}_k(\mathcal{G})$ the \newterm{$k$th Chern character} of $\mathcal{G}$.
\end{definition}

\begin{example}\label{ex: H} From~\Cref{ex: Atiyah of direct sum}, we have
    \[
    \op{ch}(\mathcal O(1)) = \op{exp}(H) 
    \]
    where $H$ is the \v{C}ech cocycle given by $\left(\tfrac{dx_t}{x_t} - \tfrac{dx_u}{x_u}\right)_{tu} = \op{at}((\mathcal O(1))$.
\end{example}

Using the structure of the Atiyah class in our setting, we can decompose the Chern character into products of two distinguished cohomology classes represented by the following \v{C}ech cocycles. Write
$$
\Theta :=  \left(\left(\Lambda_{\mathbf{d}}\frac{\partial_{i}A}{x_{t}Q_i}dx_{t}  - \Lambda_{\mathbf{d}}\frac{\partial_{j}A}{x_uQ_j} dx_{u} \right)B \right)_{ij,tu} \in H^1(X, \End(\mathcal{\E}_0) \otimes \Omega_X^1)
$$
and 
$$
\Xi := \left(\frac{dA \partial_{i}B\partial_{j}AB}{Q_i Q_j}\right)_{ij, tu} \in H^1(X, \End(\mathcal{\E}_0) \otimes \Omega_X^1).
$$
\begin{remark}
Observe that $\Xi$ does not depend on the choices of $t, u$. 
Hence, as a cohomology class, we may also regard $\Xi$ as a \v{C}ech cocycle in the less refined open cover 
\begin{equation}\label{eq: U tilde cover}
\widetilde{\mathcal{U}} = \{ D(\partial_j Q) \ | \ j \in \{0, \dots, n+1\}\}
\end{equation}
of $X$.  We do this frequently. Note that this is a cover of both $X$ and $\mathbb P^{n+1}$, as $X$ is smooth.  
\end{remark}

We can check that $\Theta$ and $\Xi$ are indeed $d_{\check{C}}$-closed. From the definitions, we obtain
\begin{equation}\label{d cech T1}
d_{\check{C}} (\Theta) = 0
\end{equation}
and on the cover $\widetilde{\mathcal{U}}$, we get
\begin{equation}\label{d cech T2}
d_{\check{C}} (\Xi) = \left( \frac{dQ}{Q_{i}Q_{j}Q_{k}}\cdot B\partial_{i}A\partial_{j}B \partial_{k}A\right)_{ijk}
\end{equation}
which vanishes on $X$ as $dQ \equiv 0$ on $X$.

\begin{theorem}\label{thm: chern character}
    The $k$th Chern character $\op{ch}_k(\F)$ of $\mathcal{F}$ is the \v{C}ech cocycle given by
$$
\op{ch}_k(\mathcal{F}) = \frac{1}{k!} \op{tr}((\Theta - \Xi)^k).
$$
\end{theorem}

\begin{proof}
We compute using the Atiyah class in Proposition~\ref{prop: atiyah class}. Consider the open set 
$$
U_{(i_0,t_0), \dots, (i_k, t_k)} := \bigcap_{j=0}^{k} U_{i_j, t_j}.
$$
Now,
 \begin{align*}
 \op{ch}_k(\F)_{U_{(i_0,t_0), \dots, (i_k, t_k)}} &= \frac{1}{k!} \op{tr}\left[(\pi\otimes \id)\left(
      \Lambda_{\mathbf{d}} \frac{dx_{t_0} }{x_{t_0}}  s_{i_0} - \Lambda_{\mathbf{d}}\frac{d x_{t_1}}{x_{t_1}} s_{i_1}  -    dA\frac{\partial_{i_0}B}{Q_{i_0}} s_{i_1}  \right) \cdots \right.\\
      &\qquad \qquad \left.(\pi\otimes \id)\left(
      \Lambda_{\mathbf{d}} \frac{dx_{t_{k-1}} }{x_{t_{k-1}}}  s_{i_{k-1}} - \Lambda_{\mathbf{d}}\frac{d x_{t_k}}{x_{t_k}} s_{i_k}  -    dA\frac{\partial_{i_{k-1}}B}{Q_{i_{k-1}}} s_{i_k}  \right)\right] \\
 &= \frac{1}{k!} \op{tr}\left[\left(
      \Lambda_{\mathbf{d}} \frac{dx_{t_0} }{x_{t_0}}  s_{i_0} - \Lambda_{\mathbf{d}}\frac{d x_{t_1}}{x_{t_1}} s_{i_1}  -    dA\frac{\partial_{i_0}B}{Q_{i_0}} s_{i_1}  \right) \cdots \right.\\
      &\qquad \qquad \left.(\pi\otimes \id)\left(
      \Lambda_{\mathbf{d}} \frac{dx_{t_{k-1}} }{x_{t_{k-1}}}  s_{i_{k-1}} - \Lambda_{\mathbf{d}}\frac{d x_{t_k}}{x_{t_k}} s_{i_k}  -    dA\frac{\partial_{i_{k-1}}B}{Q_{i_{k-1}}} s_{i_k}  \right)\pi\right] \\
    &= \frac{1}{k!} \op{tr}\left[\left(
      \Lambda_{\mathbf{d}} \frac{dx_{t_0} }{x_{t_0}} \frac{\partial_{i_0} AB}{Q_{i_0}} - \Lambda_{\mathbf{d}}\frac{d x_{t_1}}{x_{t_1}} \frac{\partial_{i_1} AB}{Q_{i_1}}  -    dA\frac{\partial_{i_0}B\partial_{i_1} AB}{Q_{i_0}Q_{i_1}}  \right) \cdots \right.\\
   &\qquad \qquad \left.\left(
      \Lambda_{\mathbf{d}} \frac{dx_{t_{k-1}} }{x_{t_{k-1}}}  \frac{\partial_{i_{k-1}} AB}{Q_{i_{k-1}}} - \Lambda_{\mathbf{d}}\frac{d x_{t_k}}{x_{t_k}} \frac{\partial_{i_k} AB}{Q_{i_k}}  -    dA\frac{\partial_{i_{k-1}}B\partial_{i_k} AB}{Q_{i_{k-1}} Q_{i_k}}  \right)\right] \\
       &= \frac{1}{k!} \op{tr}((\Theta - \Xi)^k).
 \end{align*}
 The first line uses the definition of the Alexander-\v{C}ech-Whitney product (see, e.g., \cite[\textsection 4.2.1]{CKK}) to expand the expression for the exponential. The second equality uses cyclic invariance of the trace.  The third equation is \Cref{lem: splitting}.  The fourth equality uses the definitions of $\Theta$ and $\Xi$ and recompresses them using the Alexander-\v{C}ech-Whitney product.
\end{proof}

\section{The primitive component of the Chern character}

\subsection{Primitive cohomology of projective hypersurfaces}
\label{sec: jacobian formula}

  \Cref{thm: chern character} gives an explicit formula for the Chern character of an ACM sheaf in terms of the 2-periodic resolution.  In this section we setup the machinery to go from a \v{C}ech cocycle to an element of the Jacobian ring.  This will lead us to a formula in the Jacobian ring for the primitive part of the Chern character map (see \Cref{thm: KL}).

To begin, recall the following celebrated theorem of Griffiths \cite{Gri} which interprets the cohomology of a hypersurface in terms of the Jacobian ring:

\begin{theorem}[Griffiths Residue Theorem]\label{Griffiths Residue Theorem}
    There is a (vector space) isomorphism between the $(n-k,k)$-bigraded primitive cohomology of $X$ and $((k+1)\deg f - (n+2))$-graded pieces of the Jacobian ring. That is, 
    \begin{align*}
   \left(\mathbb{C}[x_0, \dots, x_{n+1}]/J(Q)\right)_{(k+1)\deg Q -(n+2)} &\stackrel{\sim}{\longrightarrow} H^{n-k,k}_{\op{prim}}(X); \\
   M &\longmapsto \op{res} \Omega^M,
    \end{align*}
    where $$\Omega^M = \frac{M\Omega}{Q^{k+1}}, \qquad \Omega =  \sum_{j=0}^{n+1} (-1)^j x_j dx_0\wedge \dots \wedge \widehat{dx_j} \wedge \cdots \wedge dx_{n+1}.$$
\end{theorem}

In this section, we explain how to invert this isomorphism at the level of \v{C}ech cohomology: start with a \v{C}ech cocycle in $H^n(X, \Omega^n_X)$ and give its corresponding element in the Jacobian ring. This is an inversion of a result of Carlson and Griffiths who give the \v{C}ech cocycle corresponding to any element of the Jacobian ring \cite[Page 7]{CG}. 
We begin by reviewing their result.

Take $dV = dx_0\wedge \dots \wedge dx_{n+1}$ and $E$ the Euler vector field. Note that $$\Omega = EdV.$$
Given $J=\{j_0, \dots, j_k\} \subseteq\{0, \dots, n+1\}$, write $\Omega_J := K_J \Omega $, where $K_J = K_{j_0} \cdots K_{j_k}$ and $K_{j}$ is the contraction by the vector field $\tfrac{\partial}{\partial x_j}$.
Consider the open cover $\widetilde{\mathcal U}$ of $X$ given in \eqref{eq: U tilde cover}.

\begin{proposition}[Carlson, Griffiths]\label{carlsongriffiths cech}
    Suppose $M \in \left(\mathbb{C}[x_0, \dots, x_{n+1}]/J(f)\right)_{(k+1)\deg f -(n+2)}$, then one can write $\op{res} \Omega^M$ as the \v{C}ech cocycle
    $$
    \op{res}\Omega^M = c_k \left(\frac{M\Omega_J}{Q_{j_0} \cdots Q_{j_k}}  \right)_{J=\{j_0, \dots, j_k\}}  \in H^k(X, \Omega_X^{n-k}),
    $$
where $c_k = \frac{(-1)^{n+k(k+1)/2}}{k!}$.
\end{proposition}

For each $\ell$, the short exact sequence
\begin{equation}\label{eqn:Poincare}
0 \longrightarrow \Omega_{X}^{n-\ell-1}(-\deg Q) \stackrel{dQ}{\longrightarrow} \Omega^{n-\ell}_{\mathbb P^{n+1}}|_X \longrightarrow \Omega_X^{n-\ell} \longrightarrow 0
\end{equation}
induces a connecting homomorphism
$$
\delta_{\ell}: H^{\ell}(X,\Omega_X^{n-\ell}) \longrightarrow H^{\ell+1}(X, \Omega_{X}^{n-\ell-1}(-\deg Q)).
$$
Fix $M \in \left(\mathbb{C}[x_0, \dots, x_{n+1}]/J(f)\right)_{(k+1)\deg f -(n+2)}$. Consider the cohomology class
$$
\Omega^M_{\ell} := c_k \left(\frac{M\Omega_J}{\partial_{j_0} Q \cdots \partial_{j_\ell}Q}  \right)_{J=\{j_0, \dots, j_\ell\}}  \in H^\ell(X, \Omega_X^{n-\ell}(-(\ell-k)\deg Q)).
$$

\begin{lemma}\label{lemma: connecting hom}
The connecting homomorphism applied to the \v{C}ech residue is precisely 
\[
\delta_{\ell} (\Omega^M_{\ell}) = \Omega^M_{\ell+1}.
\]
\end{lemma}

\begin{proof}
We first view $\Omega_{\ell}^M$ as an element in $H^{\ell}(X, \Omega_{\mathbb P^{n+1}}^{n-\ell}|_X(-(\ell-k)\deg Q))$. By the snake lemma, we know that 
$$
d_{\check{C}} (\Omega_{\ell}^M) = dQ \wedge \omega
$$
for some $\omega \in  H^{\ell+1}(X, \Omega_X^{n-\ell-1}(-(\ell+1-k)\deg Q))$ and this determines $\delta_\ell(\Omega_\ell^M)=\omega$. The claim then reduces to showing that $\omega = \Omega^M_{\ell+1}$. To condense notation, we write $dx_S:= dx_{s_0} \wedge \cdots \wedge dx_{s_k}$ if $S = \{s_0, \dots, s_k\}$ and $s_{i} < s_{i+1}$.  We compute
\begin{align*}
   \frac{1}{c_k}(d_{\check{C}} (\Omega_{\ell}^M))_{S = \{s_0, \dots, s_k\}} &=  \sum_{i=0}^k (-1)^i \frac{Q_{s_i} M K_{S \setminus\{s_i\}} E dV}{Q_{s_0} \cdots Q_{s_k}} \\
   &= \sum_{i=0} (-1)^{i+k} \frac{Q_{s_i} ME  K_{S \setminus\{s_i\}} dV}{Q_{s_0} \cdots Q_{s_k}} \\
   &=  \sum_{i=0}^k (-1)^{i+ k- s_i + \sum_{j=0}^k s_j } \frac{Q_{s_i} ME  dx_{(S \setminus\{s_i\})^{c}}}{Q_{s_0} \cdots Q_{s_k}}\\
    &=  \sum_{i=0}^k (-1)^{k + \sum_{j=0}^k s_j } \frac{Q_{s_i} ME dx_{s_i}\wedge dx_{S^{c}}}{Q_{s_0} \cdots Q_{s_k}}.
\end{align*}
Apply the Euler vector field and then add and subtract an extra term to simplify.
\begin{equation}\begin{aligned}\label{residue dQ computation}
   \frac{1}{c_k}(d_{\check{C}} &(\Omega_{\ell}^M))_{S = \{s_0, \dots, s_k\}} = \sum_{i=0}^k (-1)^{k+ \sum_{j=0}^k s_j} \frac{ x_{s_i} Q_{s_i} M dx_{S^c}}{Q_{s_0} \cdots Q_{s_k}} + \sum_{a \notin S} (-1)^{k+ \sum_{j=0}^k s_j} \frac{M Q_a dx_a \wedge Edx_{S^c}}{Q_{s_0} \cdots Q_{s_k}} \\ 
   &\quad -  \sum_{i=0}^k (-1)^{k+ \sum_{j=0}^k s_j} \frac{ MQ_{s_i} dx_{s_i} \wedge Edx_{S^c}}{Q_{s_0} \cdots Q_{s_k}} - \sum_{a \notin S} (-1)^{k+ \sum_{j=0}^k s_j} \frac{M Q_a dx_a \wedge Edx_{S^c}}{Q_{s_0} \cdots Q_{s_k}}.
\end{aligned}\end{equation}
Note that, for $a \notin S$, $dx_a \wedge E dx_{S^c} = x_a dx_{S^c}$. Thus we reduce ~\eqref{residue dQ computation} to
\begin{equation}\begin{aligned}
   \frac{1}{c_k}(d_{\check{C}} (\Omega_{\ell}^M))_{S = \{s_0, \dots, s_k\}} &=  (-1)^{k+ \sum_{j=0}^k s_j} \frac{ \sum_{i=0}^{n+1}x_{i} Q_{i} M dx_{S^c}}{Q_{s_0} \cdots Q_{s_k}}   -  (-1)^{k+ \sum_{j=0}^k s_j} \frac{ \sum_{i=0}^{n+1} MQ_{i} dx_{i} \wedge Edx_{S^c}}{Q_{s_0} \cdots Q_{s_k}} \\
   &= (-1)^{k+ \sum_{j=0}^k s_j} \frac{ rQ M dx_{S^c}}{Q_{s_0} \cdots Q_{s_k}}   -  (-1)^{k+ \sum_{j=0}^k s_j} \frac{  M dQ \wedge Edx_{S^c}}{Q_{s_0} \cdots Q_{s_k}}.
\end{aligned}\end{equation}
Since we are on $X$, the first summand vanishes and we are left with
\begin{align*}
    (d_{\check{C}} (\Omega_{\ell}^M))_{S = \{s_0, \dots, s_k\}} &=  (-1)^{k+1+ \sum_{j=0}^k s_j} \frac{  M dQ \wedge Edx_{S^c}}{Q_{s_0} \cdots Q_{s_k}}\\
    &= (-1)^{k+1} \frac{  M dQ \wedge EK_SdV}{Q_{s_0} \cdots Q_{s_k}} \\
    &= \frac{M dQ \wedge K_S E dV}{Q_{s_0} \cdots Q_{s_k}} \\
    &= dQ \wedge \frac{M  K_S E dV}{Q_{s_0} \cdots Q_{s_k}},
\end{align*}
proving the claim.
\end{proof}

Now consider the short exact sequence
$$
0 \longrightarrow \O_{X}(-(k+1)m) \stackrel{Q}{\longrightarrow} \O_{\mathbb P^{n+1}}(-km) \longrightarrow \O_X(-km) \longrightarrow 0,
$$
which induces the connecting homomorphism
\begin{equation}\label{eq: conn hom for the projective}
\delta_{\mathbb P^{n+1}}: H^{n}(X,\O(-km)) \longrightarrow H^{n+1}(\mathbb P^{n+1}, \O_{\mathbb P^{n+1}}(-(k+1)m)).
\end{equation}
Moving forward, assume that $X \subseteq \mathbb{P}^{2k+1}$ is even dimensional. Then, we denote the total composition of all of these connecting homomorphisms by
\begin{equation}\label{eq: bold delta}
\pmb{\delta}:=\delta_{\mathbb P^{2k+1}}\circ\delta_{2k-1}\circ \cdots \circ \delta_k: H^{k}(X,\Omega_X^{k}) \longrightarrow H^{2k+1}(\mathbb P^{2k+1}, \mathcal O_{X}(-(k+1)m)) 
\end{equation}

\begin{lemma}
\label{lem: total of connecting homs}
For any $\op{res}\Omega^M \in H^k(X, \Omega_X^k)$, applying $\pmb{\delta}$ yields
$$
\pmb{\delta}(\op{res}\Omega^M) =(-1)^k m c_k \left(\frac{M}{Q_0 \cdots Q_{n+1}}\right).
$$
\end{lemma}
\begin{proof}
Iterating Lemma~\ref{lemma: connecting hom}, we obtain that
$$
\delta_{2k-1} \circ\cdots \circ\delta_{k} (\op{res}\Omega^M) = c_k \left(\frac{MK_J\Omega}{Q_0 \cdots \widehat Q_i \cdots Q_{n+1}}\right)_{J = \{0, \dots, \widehat{i}, \dots, 2k+1\}}.
$$
To compute the connecting homomorphism $\delta_{\mathbb P^{n+1}}$, we lift this element to $\O_{\mathbb{P}^{2k+1}}(-km)$, apply the \v{C}ech differential, and factor out $Q$.  We have
\begin{align*}
d_{\check{C}} \left(\delta_{2k-1} \circ\cdots \circ\delta_{k} (\op{res}\Omega^M)\right) &= c_k\sum_{i=0}^{2k+1}(-1)^i \frac{M K_{\{0, \dots, \widehat{i}, \dots, 2k+1\}} EdV}{Q_0 \cdots \widehat Q_i \cdots Q_{n+1} } \\ 
    &= c_k\sum_{i=0}^{2k+1}(-1)^{i+2k+1} \frac{Q_iM E K_{\{0, \dots, \widehat{i}, \dots, 2k+1\}} dV}{Q_0 \cdots Q_{n+1} } \\ 
    &= c_k\sum_{i=0}^{2k+1}(-1)^{i+2k+1 + \left(-i + \sum_{j=0}^{2k+1} j\right)} \frac{Q_iM E dx_i }{Q_0 \cdots Q_{n+1} }\\ 
    &= c_k\sum_{i=0}^{2k+1}(-1)^{i+2k+1 + \left(-i + \sum_{j=0}^{2k+1} j\right)} \frac{M x_iQ_i }{Q_0 \cdots Q_{n+1} }\\
    &= c_k\sum_{i=0}^{2k+1}(-1)^{ 1 + \left( \sum_{j=0}^{2k+1} j\right)} \frac{M x_iQ_i }{Q_0 \cdots Q_{n+1} }\\
    &= (-1)^{k} mc_k \frac{MQ}{Q_0 \cdots Q_{n+1}}
\end{align*}
as desired.
\end{proof}

\begin{corollary}
\label{cor: connecting composition}
The map $\pmb{\delta}$ in ~\eqref{eq: bold delta}
has a one-dimensional kernel generated by the hyperplane class $H^k$ and image $(S/J(Q))_{(k+1)m-n-2}^*$ under Serre duality $H^{n+1}(\mathbb P^{n+1}, \mathcal O_{X}(-(k+1)m))  \cong H^0(\mathbb P^{n+1}, \mathcal O_{X}((k+1)m-n-2))^*$.
\end{corollary}
\begin{proof}
By the Griffiths residue theorem (see \Cref{Griffiths Residue Theorem}), the source of $\pmb{\delta}$ has a basis consisting of $H^k$ and elements of the form $\op{res}\Omega^M$.
By \Cref{lem: total of connecting homs}, elements of the form $\op{res}\Omega^M$ map to the \v{C}ech cocycle $(-1)^k m c_k \left(\frac{M}{Q_0 \cdots Q_{n+1}}\right)$.
It remains to show that under  Serre duality, this corresponds to the function 
\[
(S)_{(k+1)m - n-2} \to (S/J(Q) )_{(k+1)m - n-2} \xrightarrow{\langle M, - \rangle } k
\] and that $\pmb{\delta}(H^k)=0$.

Computing Serre duality using the Grothendieck residue symbol tells us generally that if $\{D(f_1), ..., D(f_t)\}$ is a cover of $\mathbb P^{n+1}$ and we take a top \v{C}ech cohomology class $\frac{M}{f_1\cdots f_t}$ of deg $\ell$, then the Serre dual is the function
\[
(S)_{\ell-n-2} \to (S/\langle f_1, ..., f_t \rangle)_{\ell-n-2} \xrightarrow{\langle M, - \rangle } k
\]
where $\langle -, - \rangle$ is the residue pairing.   The result follows by specializing to the cover $\widetilde{\mathcal U} = \{ D(Q_0), ..., D(Q_{n+1}) \}$.

Finally, using the explicit expression for $H$ from \Cref{ex: H},  $d_{\check{C}}(H) = 0$. Using the product structure, this implies that $d_{\check{C}}(H^k) = 0$, thus $\delta_k(H^k)=0$ and hence $\pmb{\delta}(H^k)=0$.
\end{proof}

\subsection{A polynomial formula for the primitive component of the Chern character}
In this section, we prove the following. 

\begin{theorem}\label{thm: KL} 
Let $\mathcal{G}$ be a coherent sheaf on a smooth $2k$-dimensional projective hypersurface $X$ with resolution 
$$
\cdots \xrightarrow{A} \mathcal{E}_0(-m) \xrightarrow{B} \mathcal{E}_1 \xrightarrow{A} \mathcal{E}_0 \to \mathcal{V}_t \to \dots \to \mathcal{V}_0 \to \mathcal{G} \to 0.
$$
The projection $\op{ch}_k^{\op{prim}}(\mathcal G)$ of the $k$th Chern character to the primitive cohomology of $X$  viewed as an element of the Jacobian ring is given by the formula
$$
\op{ch}_k^{\op{prim}}(\mathcal G) = \frac{(-1)^{{k+1 \choose 2 }}k!}{m}\op{tr}( \partial_{0}A\partial_1B \cdots \partial_{2k}A\partial_{2k+1} B -  \partial_{0}B\partial_1A \cdots \partial_{2k}B\partial_{2k+1} A).
$$
\end{theorem}

Let $\F := \coker A$.  By \eqref{intro eq: chern of G}, $\op{ch}_k^{\op{prim}}(\mathcal G) = \op{ch}_k^{\op{prim}}(\mathcal F)$.  Our strategy to prove ~\Cref{thm: KL} involves computing $\pmb{\delta}(\op{ch}_k(\mathcal{F}))$ and applying ~\Cref{cor: connecting composition}. We break this computation up into a series of lemmas and give a proof of the theorem at the end of the subsection. 

First, note the trace map $\op{tr}: \mathcal{E}_0 \to \O_X$ induces a commutative diagram on cohomology
\begin{equation}\label{comm diagram d cech}
\begin{tikzcd}
H^{\ell}(X, \End(\mathcal{E}_0) \otimes \Omega_X^{n-\ell}) \ar[r, "\id \otimes \delta_\ell"]\ar[d, "\op{tr}"] & H^{\ell+1}(X, \End(\mathcal{E}_0) \otimes \Omega_X^{n-\ell-1}(-m)) \ar[d, "\op{tr}"] \\
H^{\ell}(X,  \Omega_X^{n-\ell})  \ar[r, "\delta_\ell"]  & H^{\ell+1}(X, \Omega_X^{n-\ell-1}(-m)).
\end{tikzcd}
\end{equation}

Consider the map
$$
\id \otimes \delta_{n-1}: H^1(X, \End(\mathcal{E}_0) \otimes \Omega^1) \to H^2(X, \End({\mathcal{E}_0})(-m))
$$
By~\eqref{d cech T1} and~\eqref{d cech T2}, 
\begin{equation}
    \id \otimes \delta_{n-1}(\Theta) = 0 \text{ and }\id \otimes \delta_{n-1}(\Xi) = \xi := \left( \frac{B\partial_{j_0}A\partial_{j_1}B \partial_{j_2}A}{Q_{j_0}Q_{j_1}Q_{j_2}} \right)_{\{j_0, j_1, j_2\}}.
\end{equation}

\begin{lemma}\label{lem: chern to s to the k}
        There is an equality of classes $$
 \delta_{2k-1} \circ \cdots \circ \delta_k(\op{ch}_k(\F)) =  \op{tr}(\xi^k) \in H^{2k}(X, \O(-km)).$$
\end{lemma}
\begin{proof}
For brevity, write $\Psi := \Theta - \Xi$ and $\tilde\delta_\ell = \id \otimes \delta_\ell$.  Using Theorem~\ref{thm: chern character} and ~\eqref{comm diagram d cech}, we see
\begin{equation}\label{first step for prop for sk}
 \delta_{2k-1} \circ \cdots \circ \delta_k(\op{ch}_k(\F)) = \op{tr}(\tilde \delta_{2k-1} \circ \cdots \circ \tilde\delta_k(\Psi^k)).
 \end{equation}
We claim  
\begin{equation}\label{eq: inductions step on tilde deltas}
\tilde \delta_{k+(\ell-1)}  \circ \cdots \circ \tilde\delta_k(\Psi^k) =  \ell! \sum_{1 \le i_1 < \dots < i_\ell \le k} \Psi^{i_1-1} \xi \Psi^{i_2-i_1-1} \xi \cdots \xi \Psi^{i_{\ell}-i_{\ell-1}-1} \xi \Psi^{k-i_\ell}
\end{equation}
and prove it by induction. The case where $\ell = 0$ is clear.  Using that $d_{\check{C}}(\Psi) = -\xi dQ$, we compute
\begin{equation}\begin{aligned}\label{base step dc}
d_{\check{C}}(\Psi^k) &= \sum_{1 \le i_1 \le k} (-1)^{i_1} \Psi^{i_1-1} d_{\check{C}}(\Psi) \Psi^{k-i_1} \\ 
    &= \sum_{1 \le i_1 \le k}(-1)^{i_1} \Psi^{i_1-1} (-\xi dQ) \Psi^{k-i_1} \\ 
    &= \sum_{1 \le i_1 \le k} dQ \Psi^{i_1-1} \xi \Psi^{k-i_1} \\  \tilde \delta_k(\Psi^k) &= \sum_{1 \le i_1 \le k}  \Psi^{i_1-1} \xi \Psi^{k-i_1},
\end{aligned}\end{equation}
proving the $\ell=1$ case.

Using \eqref{base step dc}, we prove the induction step.
{\allowdisplaybreaks\begin{align*}
   d_{\check{C}}&(\tilde \delta_{k+(\ell-1)}  \circ \cdots \circ \tilde\delta_k((\Theta-\Xi)^k)) \\&= \ell! \sum_{1 \le i_1 < \dots < i_\ell \le k}  d_{\check{C}}(\Psi^{i_1-1} \xi \Psi^{i_2-i_1-1} \xi \cdots \xi \Psi^{i_{\ell}-i_{\ell-1}-1} \xi \Psi^{k-i_\ell})\\
   &= \ell! \left[\sum_{1 \le i_1 < \dots < i_\ell \le k}\sum_{1 \le q \le i_1-1} dQ \Psi^{q-1}\xi \Psi^{i_1-q-1}\xi  \Psi^{i_2-i_1-1} \xi \cdots \xi \Psi^{i_{\ell}-i_{\ell-1}-1} \xi \Psi^{k-i_\ell}) +\right.\\
   & \qquad   \sum_{1 \le i_1 < \dots < i_\ell \le k} \sum_{p=1}^{\ell - 1} (-1)^{i_p-p} \Psi^{i_1-1} \xi \cdots \xi \left(\sum_{q = 1}^{i_{p+1}-i_{p}-1}dQ \Psi^{q-1} \xi \Psi^{i_{p+1} - i_{p} - q-1} \right) \xi \cdots \xi \Psi^{k-i_\ell} \\
    & \qquad  + \left.\sum_{1 \le i_1 < \dots < i_\ell \le k} (-1)^{i_k-k} \Psi^{i_1-1} \xi \Psi^{i_2-i_1-1} \xi \cdots \xi \Psi^{i_{\ell}-i_{\ell-1}-1} \xi \left(\sum_{q=1}^{k-i_\ell}  dQ\Psi^{q-1} \xi \Psi^{k-i_\ell-q}\right)\right] \\
    &= \ell!dQ \left[\sum_{1 \le i_1 < \dots < i_\ell \le k}\sum_{1 \le q \le i_1-1}  \Psi^{q-1} \xi \Psi^{i_1-q-1}\xi \Psi^{i_2-i_1-1} \xi \cdots \xi \Psi^{i_{\ell}-i_{\ell-1}-1} \xi \Psi^{k-i_\ell}) +\right.\\
   & \qquad   \sum_{1 \le i_1 < \dots < i_\ell \le k} \sum_{p=1}^{\ell - 1} \Psi^{i_1-1} \xi \cdots \xi \left(\sum_{q = 1}^{i_{p+1}-i_{p}-1}\Psi^{q-1} \xi \Psi^{i_{p+1} - i_{p} - q-1} \right) \xi \cdots \xi \Psi^{k-i_\ell} \\
    & \qquad  + \left.\sum_{1 \le i_1 < \dots < i_\ell \le k} \Psi^{i_1-1} \xi \Psi^{i_2-i_1-1} \xi \cdots \xi \Psi^{i_{\ell}-i_{\ell-1}-1} \xi \left(\sum_{q=1}^{k-i_\ell}  \Psi^{q-1} \xi \Psi^{k-i_\ell-q}\right)\right] \\
    &= (\ell+1)! dQ \sum_{1 \le i_1 < \dots < i_{\ell+1} \le k} \Psi^{i_1-1} \xi \Psi^{i_2-i_1-1} \xi \cdots \xi \Psi^{i_{\ell+1}-i_{\ell}-1} \xi \Psi^{k-i_{\ell+1}}.
\end{align*}}
The induction step for~\eqref{eq: inductions step on tilde deltas} then follows immediately. Taking $\ell = k$ and plugging into~\eqref{first step for prop for sk}, we conclude 
\begin{equation*}
   \delta_{2k-1} \circ\cdots \circ\delta_{k} (\op{ch}_k(\F)) = \frac{1}{k!}\op{tr}(\tilde \delta_{2k-1} \circ \cdots \circ \tilde\delta_k(\Psi^k)) = \frac{1}{k!}\op{tr} (k! \xi^k) = \op{tr}(\xi^k). \qedhere
\end{equation*} 
\end{proof}
 
We introduce the following notation for brevity.
Take $k$ to be a positive integer. We denote
\begin{align*}
\partial_{i_1,i_2, \dots, i_{2k}}^A &:= \partial_{i_1}A\partial_{i_2}B \cdots \partial_{i_{2k-1}}A\partial_{i_{2k}} B;  \qquad  \partial_{i_1,i_2, \dots, i_{2k}}^B := \partial_{i_1}B\partial_{i_2}A \cdots \partial_{i_{2k-1}}B\partial_{i_{2k}} A;  \\
\partial_{i_1,i_2, \dots, i_{2k+1}}^A &:= \partial_{i_1}A\partial_{i_2}B \cdots \partial_{i_{2k}}B\partial_{i_{2k+1}} A;  \qquad \partial_{i_1,i_2, \dots, i_{2k+1}}^B := \partial_{i_1}B\partial_{i_2}A \cdots \partial_{i_{2k-1}}A\partial_{i_{2k+1}} B.
\end{align*}

Recall our goal is to compute $\pmb{\delta}(\op{ch}_k\mathcal F)$. By Lemma~\ref{lem: chern to s to the k}, we have $\pmb{\delta}(\op{ch}_k,\mathcal F) =\op{tr}(\op{id}\otimes \delta_{\mathbb{P}^{2k+1}}(\xi^k))$.
This final connecting homomorphism $\delta_{\mathbb{P}^{2k+1}}$  lifts $\xi^k$ to  $\hat \xi^k$ on $\mathbb{P}^{2k+1}$, takes the \v{C}ech differential, and divides by $Q$. We start this computation by lifting $\xi$ to $\hat \xi$ and compute that
\begin{equation}\label{s tilde dcech}
d_{\check{C}} (\hat \xi) = \left(\frac{B\partial_{i_0, i_1, i_2, i_3}^A A - Q \partial_{i_0, i_1, i_2, i_3}^B}{Q_{i_0} Q_{i_1} Q_{i_2}Q_{i_3}}\right)_{ \{i_0, i_1, i_2, i_3\}}.
\end{equation}

\begin{lemma}\label{d cech s hat}
The following cocycles agree in cohomology
$$ \delta_{\mathbb P^{n+1}}( \op{tr}\xi^k) =  \frac{\partial^A_{0, \dots, 2k+1} - \partial^B_{0, \dots, 2k+1}}{Q_0 \cdots Q_{2k+1}}.
$$
\end{lemma}

\begin{proof}
To prove the lemma, we show that
\begin{equation}
\label{eq: cech zeta} d_{\check{C}}( \hat \xi^k) =  \frac{B\partial^A_{0, \dots, 2k+1}A - Q\partial^B_{0, \dots, 2k+1}}{Q_0 \cdots Q_{2k+1}}.
\end{equation}
up to a coboundary. On the level of cohomology, we then obtain the desired statement since
\begin{align*}
\delta_{\mathbb P^{n+1}}( \op{tr}\xi^k)   & = \op{tr}\left(\frac{AB\partial^A_{0, \dots, 2k+1} - Q\partial^B_{0, \dots, 2k+1}}{Q(Q_0 \cdots Q_{2k+1})}\right) & \text{by cyclic invariance of trace;} \\
 & = \op{tr}\left(\frac{\partial^A_{0, \dots, 2k+1} - \partial^B_{0, \dots, 2k+1}}{Q_0 \cdots Q_{2k+1}}\right) & \text{since $AB=Q$.} 
\end{align*}

To prove \eqref{eq: cech zeta}, we use induction on $r$ to show that, for $0 \le r < k$, we have that $\delta_{\mathbb{P}^{2k+1}}( \hat \xi^k)$ equals 
\begin{equation}\label{KL Claim 1}
    \sum_{\ell=1}^{k-r} \frac{\left(\prod_{1 \le p < \ell} B\partial^A_{2p, 2p-1, 2p}\right) \left[B \partial_{2\ell-2, \dots, 2\ell+2r+1}^A A - Q\partial_{2\ell-2, \dots, 2\ell+2r+1}^B \right] \left(\prod_{\ell+r < p \le k} B \partial^A_{2p-1, 2p, 2p+1}\right)}{\left(\prod_{1 \le p < \ell} Q_{2m}\right) \left(\prod_{\ell+r\le m < k} Q_{2m+1}\right) Q_0 \cdots Q_{2k+1}}.
\end{equation}
The statement of the lemma is just the case where $r = k-1$.

First, we use the cdga structure of the \v{C}ech differential and Alexander-\v{C}ech-Whitney product and \eqref{s tilde dcech} to compute
\begin{align*}
   d_{\check{C}}( \hat \xi^k) &= \sum_{\ell = 1}^k \hat \xi^{\ell-1}  d_{\check{C}}( \hat \xi) \hat \xi^{k-\ell} \\
   &= \sum_{\ell=1}^k \frac{\left(\prod_{1 \le p < \ell} B\partial^A_{2p, 2p-1, 2p}\right) \left[B \partial_{2\ell-2, 2\ell-1, 2\ell, 2\ell+1}^A A - Q\partial_{2\ell-2, 2\ell-1, 2\ell, 2\ell+1}^B \right] \left(\prod_{\ell < p \le k} B \partial^A_{2p-1, 2p, 2p+1}\right)}{\left(\prod_{1 \le p < \ell} Q_{2m}\right) \left(\prod_{\ell\le m < k} Q_{2m+1}\right) Q_0 \cdots Q_{2k+1}},
\end{align*}
thus the $r=0$ case of the claim in ~\eqref{KL Claim 1} holds. Note there are no signs in the above equation as each \v{C}ech cocycle $\hat \xi$ is even. 

When $\ell$ is even, the identity 
\begin{equation}\label{KL identity 1}
B\partial_{i_1, \dots, i_\ell}^AA -Q\partial_{i_1, \dots, i_\ell}^B= \sum_{k=1}^\ell (-1)^{k-1} Q_{i_k} \partial_{i_1, \dots, \hat i_k, \dots, i_\ell}^BA.
\end{equation}
follows from induction on $\ell$ and the identity $Q_i = \partial_iAB + A\partial_iB$.

We now induct on $r$. Suppose that ~\eqref{KL Claim 1} is true for a fixed $r$ less than $k-1$. We use~\eqref{KL identity 1} to reduce ~\eqref{KL Claim 1} to

\begin{equation}\begin{aligned}\label{eq: KL proof 1}
    \sum_{\ell=1}^{k-r}& \frac{\left(\prod_{1 \le p < \ell} B\partial^A_{2p, 2p-1, 2p}\right) \left[B \partial_{2\ell-2, \dots, 2\ell+2r+1}^A A - Q\partial_{2\ell-2, \dots, 2\ell+2r+1}^B \right] \left(\prod_{\ell+r < p \le k} B \partial^A_{2p-1, 2p, 2p+1}\right)}{\left(\prod_{1 \le p < \ell} Q_{2m}\right) \left(\prod_{\ell+r\le m < k} Q_{2m+1}\right) Q_0 \cdots Q_{2k+1}}.\\
 &=    \sum_{\ell=1}^{k-r} \frac{\left(\prod_{1 \le p < \ell} B\partial^A_{2p, 2p-1, 2p}\right) \left[ \sum_{q=2\ell-2}^{2\ell+2r+1} (-1)^q Q_q \partial^B_{2\ell-2, \dots, \hat q, \dots, 2\ell+2r+1}A \right] \left(\prod_{\ell+r < p \le k} B \partial^A_{2p-1, 2p, 2p+1}\right)}{\left(\prod_{1 \le p < \ell} Q_{2m}\right) \left(\prod_{\ell+r\le m < k} Q_{2m+1}\right) Q_0 \cdots Q_{2k+1}} .
\end{aligned}\end{equation}

Now consider the quantity

$$ \frac{\left(\prod_{1 \le p < \ell} B\partial^A_{2p, 2p-1, 2p}\right) \left[  (-1)^q Q_q \partial^B_{2\ell-2, \dots, \hat q, \dots, 2\ell+2r+1}A \right] \left(\prod_{\ell+r < p \le k} B \partial^A_{2p-1, 2p, 2p+1}\right)}{\left(\prod_{1 \le p < \ell} Q_{2m}\right) \left(\prod_{\ell+r\le m < k} Q_{2m+1}\right) Q_0 \cdots Q_{2k+1}}$$
when $q \ne 2\ell-2, 2\ell+2r+1$, which equals

$$
 \frac{\left(\prod_{1 \le p < \ell} B\partial^A_{2p, 2p-1, 2p}\right) \left[  (-1)^q  \partial^B_{2\ell-2, \dots, \hat q, \dots, 2\ell+2r+1}A \right] \left(\prod_{\ell+r < p \le k} B \partial^A_{2p-1, 2p, 2p+1}\right)}{\left(\prod_{1 \le p < \ell} Q_{2m}\right) \left(\prod_{\ell+r\le m < k} Q_{2m+1}\right) Q_0 \cdots \widehat{Q_q} \cdots Q_{2k+1}}.
$$
The denominator has no $Q_q$ term, hence it is a coboundary and vanishes in cohomology.  Thus~\eqref{eq: KL proof 1} reduces to
 
 \begin{equation}\begin{aligned}
 & \sum_{\ell=2}^{k-r} \frac{\left(\prod_{1 \le p < \ell} B\partial^A_{2p, 2p-1, 2p}\right) \left[   Q_{2\ell+2} \partial^B_{2\ell-1, \dots, 2\ell+2r+1}A \right] \left(\prod_{\ell+r < p \le k} B \partial^A_{2p-1, 2p, 2p+1}\right)}{\left(\prod_{1 \le p < \ell} Q_{2m}\right) \left(\prod_{\ell+r\le m < k} Q_{2m+1}\right) Q_0 \cdots Q_{2k+1}} \\
 & \quad - \sum_{\ell=1}^{k-r-1} \frac{\left(\prod_{1 \le p < \ell} B\partial^A_{2p, 2p-1, 2p}\right) \left[   Q_{2\ell+2r+1} \partial^B_{2\ell-2, \dots, 2\ell+2r}A \right] \left(\prod_{\ell+r < p \le k} B \partial^A_{2p-1, 2p, 2p+1}\right)}{\left(\prod_{1 \le p < \ell} Q_{2m}\right) \left(\prod_{\ell+r\le m < k} Q_{2m+1}\right) Q_0 \cdots Q_{2k+1}} \\
\end{aligned} \end{equation}
Continuing, the quantity above equals
\begin{equation}\begin{aligned}
 & \sum_{\ell=2}^{k-r} \frac{\left(\prod_{1 \le p < \ell-1} B\partial^A_{2p, 2p-1, 2p}\right) \left[  B\partial^B_{2\ell-4, \dots, 2\ell+2r+1}A \right] \left(\prod_{\ell+r < p \le k} B \partial^A_{2p-1, 2p, 2p+1}\right)}{\left(\prod_{1 \le p < \ell-1} Q_{2m}\right) \left(\prod_{\ell+r\le m < k} Q_{2m+1}\right) Q_0 \cdots Q_{2k+1}} \\
 & \quad - \sum_{\ell=1}^{k-r-1} \frac{\left(\prod_{1 \le p < \ell} B\partial^A_{2p, 2p-1, 2p}\right) \left[    \partial^B_{2\ell-2, \dots, 2\ell+2r}AB\partial^{A}_{2\ell+2r+1,\dots, 2\ell+2r+3} \right] \left(\prod_{\ell+r+1 < p \le k} B \partial^A_{2p-1, 2p, 2p+1}\right)}{\left(\prod_{1 \le p < \ell} Q_{2m}\right) \left(\prod_{\ell+r+1\le m < k} Q_{2m+1}\right) Q_0 \cdots Q_{2k+1}} \\
  &\qquad = \sum_{\ell=1}^{k-r-1} \frac{\left(\prod_{1 \le p < \ell} B\partial^A_{2p, 2p-1, 2p}\right) \left[  B\partial^B_{2\ell-2, \dots, 2\ell+2r+3}A \right] \left(\prod_{\ell+r+1 < p \le k} B \partial^A_{2p-1, 2p, 2p+1}\right)}{\left(\prod_{1 \le p < \ell} Q_{2m}\right) \left(\prod_{\ell+r+1\le m < k} Q_{2m+1}\right) Q_0 \cdots Q_{2k+1}} \\
 &\qquad \quad - \sum_{\ell=1}^{k-r-1} \frac{\left(\prod_{1 \le p < \ell} B\partial^A_{2p, 2p-1, 2p}\right) \left[   Q \partial^B_{2\ell-2, \dots, 2\ell+2r+3} \right] \left(\prod_{\ell+r+1 < p \le k} B \partial^A_{2p-1, 2p, 2p+1}\right)}{\left(\prod_{1 \le p < \ell} Q_{2m}\right) \left(\prod_{\ell+r+1\le m < k} Q_{2m+1}\right) Q_0 \cdots Q_{2k+1}} \\
\end{aligned}\end{equation}
After combining the two sums in the last line, we get ~\eqref{KL Claim 1} with $r+1$ instead of $r$. This proves the claim and hence the lemma.
\end{proof}
We now prove Theorem~\ref{thm: KL}.

\begin{proof}[Proof of Theorem~\ref{thm: KL}]
Combining our previous results, we have
\begin{align*}
\pmb{\delta}(\op{ch}_k(\F)) & = 
\delta_{\mathbb P^{n+1}}(\op{tr}(\xi^k)) & \text{by \Cref{lem: chern to s to the k};} \\
& = \op{tr}\left(\frac{\partial^A_{0, \dots, 2k+1} - \partial^B_{0, \dots, 2k+1}}{Q_0 \cdots Q_{2k+1}}\right) & \text{by \Cref{d cech s hat};  } \\
& = \pmb{\delta} \left(\frac{(-1)^k}{mc_k}\op{res}\Omega^{\op{tr}(\partial^A_{0, \dots, 2k+1} - \partial^B_{0, \dots, 2k+1})}\right) & \text{ by \Cref{lem: total of connecting homs}}.
\end{align*}
Therefore, by \Cref{cor: connecting composition}, 
\[
\op{ch}_k^{\op{prim}}(\F) = \frac{(-1)^k}{mc_k}\op{res}\Omega^{\op{tr}(\partial^A_{0, \dots, 2k+1} - \partial^B_{0, \dots, 2k+1})}
\]
and the result follows from the Griffiths residue theorem (see \Cref{Griffiths Residue Theorem}).
\end{proof}

We now prove that this theorem yields an equivalent formulation of the Hodge conjecture.

\begin{proof}[Proof of \Cref{intro cor: HC for Proj Hypers}]


If $f$ is Hodge and
\[
  f = \sum_{i=1}^N c_i \left( \op{tr}( \partial_{0}A_i\partial_1B_i \cdots \partial_{2k}A_i\partial_{2k+1} B_i -  \partial_{0}B_i\partial_1A_i \cdots \partial_{2k}B_i\partial_{2k+1} A_i)\right)
\] 
then by  \Cref{thm: KL}
\[
  f = \sum_{i=1}^N c'_i  \op{ch}_k^{\op{prim}}(\mathcal F_i)
\]
for some nonzero $c_i'\in \mathbb Q$, where $\mathcal F_i$ is the sheaf associated to the cokernel of $A_i$ when viewed as a morphism of graded free $R$-modules.

Conversely, if 
\[
  f = \sum_{i=1}^N c'_i  \op{ch}_k^{\op{prim}}(\mathcal G_i)),
\]
then we set $A_i,B_i$ to be the two periodic tail of $\Gamma_*(\mathcal G_i)$ in Eisenbud's resolution \cite{Eisenbud}.
\end{proof}

\begin{remark} In \cite{BrownWalker}, Brown and Walker show that the Hodge conjecture for the hypersurface $X$ is equivalent to the noncommutative Hodge conjecture for the category of singularities or equivalently to the noncommutative Hodge conjecture for the category of matrix factorizations \cite{Orlov2}.  \Cref{intro cor: HC for Proj Hypers} can be thought of as an explicit realization of this.  
\end{remark}

\subsection{Chern characters for complete intersections}

The most basic matrix factorization we can construct is called the \newterm{Koszul Factorization}. This factorization resolves the ACM sheaf $\mathcal{F}$ associated to the structure sheaf $\O_Z$ of a complete intersection $Z$. In this section, we find that ~\Cref{thm: KL} has a particularly nice form when applied to Koszul factorizations. Indeed, we recover a formula due to Villaflor \cite{loyola2024toric} for the primitive part of the Chern character of $\O_Z$. 

  To begin, we setup some notation and observations about basic operators on the exterior algebra $\Lambda^\bullet V$ on $V = k^{r}$  with basis $e_1, ..., e_r$.  For $1 \leq i \leq r$, denote the wedge products by
\begin{align*}
\theta_i^\ell : \Lambda^\ell V & \to \Lambda^{\ell+1} V \\
v & \mapsto v \wedge e_i
\end{align*}
Identifying $V = V^*$ using the basis $e_1, ..., e_r$, we also consider the contraction maps and index them with the continued enumeration
\begin{align*}
\theta^\ell_{p+r} : \Lambda^\ell V & \to \Lambda^{\ell-1} V \\
e_{i_1} \wedge ... \wedge e_{i_\ell} & \mapsto \begin{cases}
(-1)^{\ell-p} e_{i_1} \wedge ... \wedge \widehat{e_{i_p}} \wedge ... \wedge e_{i_\ell} & \text{ if } i_p \in \{i_1, ..., i_\ell \} \\
0 & \text{ if } i_p \notin \{i_1, ..., i_\ell \} 
\end{cases}
\end{align*}
where $i_1 < ... < i_\ell$.  
By summing these operators, we obtain new operators
\begin{align*}
\theta_i^0 &:= \sum_{\ell \equiv 0 \mod 2} \theta_i^\ell: \Lambda^{\text{even}} V \to \Lambda^{\text{odd}} V,\\
\theta_i^1 &:= \sum_{\ell \equiv 1 \mod 2} \theta_i^\ell: \Lambda^{\text{odd}} V \to \Lambda^{\text{even}} V,
\end{align*}
and
\[
\theta_i := \sum_{\ell} \theta_i^\ell = \theta_i^0+\theta_i^1: \Lambda^{\bullet} V \to \Lambda^{\bullet} V.
\]
For a composition $\theta_{i_1} \cdots \theta_{i_{2p}}$, we denote by 
\[
\op{str}(\theta_{i_1} \cdots \theta_{i_{2p}}) = \op{tr}(\theta_{i_1}^1 \cdots \theta_{i_{2p}}^0) - \op{tr}(\theta_{i_1}^0 \cdots \theta_{i_{2p}}^1) 
\]
the supertrace of the operator $\theta_{i_1} \cdots \theta_{i_{2p}}$.  We first prove the following basic result about this supertrace.
\begin{lemma}
\label{lem: trace of exterior operators}
We have
\[
\op{str}(\theta_{i_1} \cdots \theta_{i_{2p}}) 
= \begin{cases}
(-1)^r\op{sgn}(\sigma) & \text{ if $p=r$ and $i_1, \dots i_{2p}$ all distinct;}  \\
0 & \text{ if $p < r$;} 
\end{cases}
\]
where $\op{sgn}(\sigma)$ is the sign of the permutation $\sigma$ sending $\{i_1, \dots, i_\ell\}$ to $\{1, \dots, 2r\}$. 
\end{lemma}

\begin{proof}
Under the identification 
\[
\Lambda^\bullet V \cong (\Lambda^\bullet k)^{\otimes r} 
\]
an operator which does not use every index $i$, will either be an odd endomorphism, or the identity on some tensor summand.  Either way, the supertrace is zero.  Hence, to get something non-zero we must use every index.  This handles the case where $p < r$.

We now focus on the $p =r$ case. First, observe that
\begin{equation}\label{eq: koszul relations}
\theta_i\theta_j+\theta_j\theta_i = c_{ij}\id
\text{ and } \theta_i\theta_j+\theta_j\theta_i = c_{ij}\id,
\end{equation}
where 
\[
c_{ij} = \begin{cases}
   1 & \text{ if $|i-j| = r$},\\
   0 & \text{ otherwise}.
\end{cases}
\]
This allows us to reorder our operator with each permutation introducing a sign, so that 
\[
\theta_{i_1} \cdots \theta_{i_{2r}} = \op{sgn}(\sigma) \theta_{1} \cdots\theta_{2r} + \text{(lower order terms)}.
\]
Next, note that 
$$
\theta_{1} \cdots\theta_{2r}(e_1 \wedge\cdots \wedge e_r) = e_1 \wedge\cdots \wedge e_r; \quad \theta_{1} \cdots\theta_{2r}(e_{i_1} \wedge\cdots \wedge e_{i_\ell}) = 0 \text{ if $l<r$}, 
$$
hence $\op{str}(\theta_{1} \cdots\theta_{2r}) = (-1)^r$.
Applying supertrace and using the case where $p<r$, we get
\begin{align*}
\op{str}(\theta_{i_1} \cdots \theta_{i_{2r}}) & = \op{sgn}(\sigma)\op{str}(\theta_{1} \cdots\theta_{2r}) \\
& = (-1)^r\op{sgn}(\sigma),
\end{align*}
as desired.
\end{proof}

Suppose
\[
Q = \sum_{i = 1}^r a_ib_i.
\]
That is, let $t = \lfloor \frac{r}{2} \rfloor$ and
\[
E_0 :=  \bigoplus_{j=0}^{t} \Lambda^{2j} V \otimes_k R  \ \text{ and } E_1 :=  \bigoplus_{j=0}^{t} \Lambda^{2j+1} V \otimes_k R.
\]

The differential on the Koszul factorization is then given explicitly as
\[
\cdots \xrightarrow{A} E_0(-m) \xrightarrow{B} E_1 \xrightarrow{A} E_0, 
\]
where 
\begin{equation}
\label{eq: Kos diff}
A = \sum_{\substack{1 \leq i \leq r,\\ \ell \equiv 1 \pmod 2}} a_i \theta^\ell_i + b_i \theta^\ell_{i+r} \text{ and } B = \sum_{\substack{1 \leq i \leq r, \\ \ell \equiv  0 \pmod 2}} a_i \theta^\ell_i + b_i \theta^\ell_{i+r}.
\end{equation}
\begin{remark}
 Let $\mathbf a := (a_1, .., a_r)$ and $\mathbf b := (b_1,..., b_r)$.  Then the matrices $A,B$ in the Koszul factorization can be written in the condensed notation
\[
A = B = \cdot \wedge \mathbf a + \lrcorner \mathbf{b}.
\]
\end{remark}
When $a_1, ..., a_r$ form a regular sequence on $S$, the Koszul factorization is the resolution of the top syzygy of $R/\langle a_1, ..., a_r \rangle$ (see \cite{Tate, Eisenbud}).  Geometrically, this means that the associated sheaves
\[
\cdots \xrightarrow{A} \mathcal E_0(-m) \xrightarrow{B} \mathcal E_1 \xrightarrow{A} \mathcal E_0 
\]
resolve the ACM sheaf associated to the structure sheaf $\mathcal O_Z$ of the complete intersection $Z := Z(a_1, .., a_r) \subseteq \mathbb P^{n+1}$.

\begin{theorem}[Villaflor]\label{thm: koszul det chprim}
Suppose that $Q = \sum_{i=0}^k a_ib_i$ and that $Z = Z(a_0, ..., a_{k})$ is a complete intersection in $\mathbb P^{2k+1}$.  Consider the matrix
\[
M_Z := \begin{bmatrix}
\partial_0 a_0 & \cdots& \partial_0 a_k & \partial_0 b_0 & \cdots & \partial_0 b_k  \\
\vdots & &\vdots & \vdots & & \vdots \\
\partial_{2k+1} a_0 & \cdots & \partial_{2k+1} a_k &  \partial_{2k+1} b_0 &  \cdots & \partial_{2k+1} b_k  \\
\end{bmatrix}
\]
Then
\[
\op{ch}^{\op{prim}}_k(\mathcal O_Z) = \frac{(-1)^{{k \choose 2}+1}k!}{m}\op{det}(M_Z).
\]
Furthermore for a Koszul factorization of rank greater than $k+1$, we have $\op{ch}^{\op{prim}}_k(\mathcal O_Z) = 0$. Moreover, any Koszul factorization of $Q$ must have rank at least $k+1$ since $X = Z(Q)$ is smooth.
\end{theorem}
\begin{proof}
Renumerate 
\[
z_i := \begin{cases} a_i & \text{ if $0 \leq i \leq k$}\\
b_{i-k-1} & \text{ if $k+1 \leq i \leq 2k+1$}.
\end{cases}
\]
Then, using \Cref{thm: KL} and \Cref{lem: trace of exterior operators}, we have
\begin{align*}
\op{ch}^{\op{prim}}_k(\mathcal O_Z)  & = \frac{(-1)^k}{mc_k} \op{tr}(\partial_0 A \partial_1 B  \cdots   \partial_{2k+1} A \partial_{2k+1} B - \partial_0 B \partial_1 A  \cdots    \partial_{2k+1} B \partial_{2k+1} A) \\
& = \frac{(-1)^k}{mc_k} \sum_{0 \leq i_0, ...,i_{2k+1} \leq k} 
\partial_0z_{i_0} \cdots \partial_{2k+1}z_{i_{2k+1}} \op{tr}(\theta^1_{i_0} \cdots \theta^0_{i_{2k+1}})- \partial_0z_{i_0} \cdots \partial_{2k+1}z_{i_{2k+1}} \op{tr}(\theta^0_{i_0} \cdots \theta^1_{i_{2k+1}}) \\
& = \frac{(-1)^k}{mc_k} \sum_{0 \leq i_0, ...,i_{2k+1} \leq k} 
\partial_0z_{i_0} \cdots \partial_{2k+1}z_{i_{2k+1}} \op{str}(\theta_{i_0} \cdots \theta_{i_{2k+1}})\\
& = \frac{(-1)^k}{mc_k}\sum_{0 \leq i_0 \neq ... \neq i_{2k+1} \leq k}  (-1)^{k+1} \op{sgn}(\sigma) \partial_0z_{i_0} \cdots \partial_{2k+1}z_{i_{2k+1}} \\
& = \frac{-1}{mc_k}\op{det}(M_Z).
\end{align*}

The case where $r > k+1$ follows easily from the same computation using \Cref{lem: trace of exterior operators}. If $r < k+1$, this means $Q = \sum_{i=1}^r a_ib_i$.  Then,
\[
\partial_jQ = \sum_{i=1}^r(\partial_ja_i)b_i +a_i(\partial_jb_i) \in \langle a_1,..., a_r,b_1, ..., b_r \rangle \text{ for all }j.
\]
Hence $X_{\op{sing}} = Z(Q_0,...,Q_{2k+1}) \subseteq Z(a_1, ..., a_r,b_1,...,b_r)$ and $\op{dim} Z(a_1, ..., a_r,b_1,...,b_r) > 0$ since $r<k+1$.  This contradicts the assumption that $X$ is smooth.
\end{proof}

\section{Algebraic cycles for Fermat hypersurfaces}

In this section, we restrict ourselves to the special case where 
$$
Q = \sum_{i=0}^{n+1} x_i^{m}
$$
which defines a Fermat hypersurface $X_m^n$. In certain cases, the Hodge conjecture has been proven for $X_m^n$. 
\begin{theorem} \label{thm: known Fermat Hodge}
The Hodge conjecture is true for $X_m^n$ when $n$ is even in the following cases:
\begin{enumerate}
    \item[(i)] $m$ is prime or a power of a prime.
    \item[(ii)] $m \le 21$ or $m = 27$.
    \item[(iii)] $n=4$ and $m$ is coprime to 6
\end{enumerate}
\end{theorem}
The case where $d$ is prime is given independently by Shioda and Ran \cite{Ran, ShiodaHodge}. When $m < 21$ this is treated in \cite{ShiodaHodge}. When $m$ is a power of a prime, this is proven by Aoki \cite{Aoki}. Case (iii) and when $m = 21, 27$ is proven by da Silva following Shioda's program \cite{daSilva}.

\subsection{Hodge classes on Fermat hypersurfaces}

Let $\mu_m$ be the group of $m$th root of unity and set $G_m^n = (\mu_m)^{n+2}/ \Delta$ where $\Delta$ is the diagonal. The group $G_m^n$ acts on $X$ naturally. Its character group $\widehat G_m^n$ can be identified with the following 
$$
\left\{ (a_0, \dots, a_{n+1})\in (\Z/m\Z)^{n+2} \ \mid \ \sum_{i=0}^{n+1} a_i = 0\right\}
$$
by taking $\alpha \in \widehat G_m^n$ to $(a_0, \dots, a_{n+1})$  if $\alpha (g) = \zeta_0^{a_0} \cdots \zeta_{n+1}^{a_{n+1}}$ when $g = (\zeta_0, \dots, \zeta_{n+1}) \in G_m^n$. 

Given $\alpha \in \widehat G_m^n$, let 
$$
V(\alpha) := \{ \xi \in H^n(X, \C) \ | \ g^*\xi = \alpha(g)\xi \text{ for all $g \in G_m^n$}\}
$$
Moreover, write
$$
\mathfrak{A}_m^n = \{(a_0, \dots, a_{n+1}) \in \widehat G_m^n \ | \ a_i \ne 0 \text{ for all $i$}\}
$$
The primitive cohomology breaks into irreducible representations of $G^n_m$ \cite{Katz, Ogus} 
$$
H^n_{\op{prim}}(X, \C) = \bigoplus_{\alpha \in \mathfrak{A}_m^n} V(\alpha), \quad \dim V(\alpha) = 1.
$$

\begin{proposition}\label{prop: shift alpha and monomial}
    Let $\alpha= (a_0, \dots, a_{n+1}) \in \mathfrak{A}_m^n$ and assume that $0 < a_i <m$. Take $M(\alpha) := x_0^{a_0-1} \cdots x_{n+1}^{a_{n+1} - 1}$. Then 
    $$
    \op{res} \Omega^{M(\alpha)} \in V(\alpha).
    $$
\end{proposition}
\begin{proof}
This follows from a straightforward computation using the \v{C}ech cocycle representation given in Proposition~\ref{carlsongriffiths cech}.   Namely, given $g = (\zeta_0, \dots, \zeta_{n+1}) \in G_m^n$ we compute
\begin{align*}
g \cdot \op{res} \Omega^{M(\alpha)} & =  c_k \left(g^*\left(\frac{M(\alpha)\Omega_J}{Q_{j_0} \cdots Q_{j_k}}\right)  \right)_{J=\{j_0, \dots, j_k\}} \\
    &= c_k \left(\frac{\zeta_0^{a_0-1} \cdots \zeta_{n+1}^{a_{n+1}}M(\alpha) \left(\prod_{i \notin J} \zeta_i\right)\Omega_J}{\left(\prod_{j \in J} \zeta_j^{-1}\right)Q_{j_0} \cdots Q_{j_k}}  \right)_{J=\{j_0, \dots, j_k\}} \\
    & = \alpha(g)\op{res}\Omega^{M(\alpha)}.
 \qedhere
\end{align*}
\end{proof}
Alternatively, given $M= x_0^{m_0}\cdots x_{2k+1}^{m_{2k+1}}$, we define $\alpha(M):= (m_0+1, \dots, m_{2k+1}+1) \in \mathfrak{A}_r^{n}$.  The following is then an easy consequence of \Cref{thm: KL}.
\begin{proposition} \label{prop: g action on MF}
Given $g \in G_m^n$, if
\[
\op{ch}^{\op{prim}}_k (\F) = \sum_{M =(m_0, \dots, m_{2k+1})} c_M x^M\in \op{Jac}(Q),
\]
then
\[
\op{ch}^{\op{prim}}_k (g^*\F) = \sum_{M =(m_0, \dots, m_{2k+1})} \alpha(M)(g)c_M x^M\in \op{Jac}(Q).
\]
\end{proposition}
\begin{proof}
If $A,B$ is the matrix factorization corresponding to the 2-periodic tail of $\mathcal F$ then $g^*A, g^*B$ is the pullback under $x_i \mapsto g_i \cdot x_i$.  Using  \Cref{thm: KL} and the chain rule, 
\begin{align*}
\op{ch}_k^{\op{prim}}(g^*\F) & = \frac{(-1)^k}{mc_k}\op{tr}( \partial_{0}g^*A\partial_1g^*B \cdots \partial_{2k}g^*A\partial_{2k+1}g^*B -  \partial_{0}B\partial_1g^*A \cdots \partial_{2k}g^*B\partial_{2k+1} A) \\
& = g_0\cdots g_{2k+1}\frac{(-1)^k}{mc_k}g^*\op{tr}( \partial_{0}A\partial_1B \cdots \partial_{2k}A\partial_{2k+1}B -  \partial_{0}B\partial_1A \cdots \partial_{2k}B\partial_{2k+1} A) \\
& = g_0\cdots g_{2k+1}\frac{(-1)^k}{mc_k}g^*\left(\sum_{M =(m_0, \dots, m_{2k+1})}c_M x^M\right) \\
& = \frac{(-1)^k}{mc_k}\left(\sum_{M =(m_0, \dots, m_{2k+1})} c_M g_0\cdots g_{2k+1}g^*x^M\right) \\
& = \frac{(-1)^k}{mc_k}\left(\sum_{M =(m_0, \dots, m_{2k+1})} c_M \alpha(M)(g)x^M\right). \qedhere
\end{align*}
\end{proof}

The following observation was used in the work of \cite{Aoki}.
\begin{proposition}\label{prop: isolate algebraics fermat}
    Let $\F$ be a coherent sheaf on $X_m^{2k}$. Suppose 
    $$
    \op{ch}^{\op{prim}}_k (\F) = \sum_{M =(m_0, \dots, m_{2k+1})} c_M x^M\in \op{Jac}(Q)
    $$
    where we sum over all $M=(m_0, \dots, m_{2k+1})$ with $0 \le m_i \le m-2$. If $c_M \ne 0$, then $\op{res} \Omega^{M} \in V(\alpha(M)) $ lies in the complexification of the image of the Chern character.
\end{proposition}
\begin{proof}

We have
\begin{align*}
\sum_{g \in G^n_m} \alpha(M)^{-1}(g)   \op{ch}^{\op{prim}}_k (g^*\F) 
& = \sum_{g \in G^n_m}  \sum_{N} c_N \alpha(M)^{-1}(g)\alpha(N)(g)x^N \\
& = \sum_{N} c_N  (\sum_{g \in G^n_m} \alpha(M)^{-1}(g)\alpha(N)(g))x^N \\
& = c_M|G^n_m|x^M.
\end{align*}
The first line is by \Cref{prop: g action on MF}. The second line rearranges the summation.  The third line follows from the fact that $\alpha(M) = \alpha(N)$ if and only $M=N$ and standard character theory.
\end{proof}

\subsection{Chern characters of complete intersections on Fermat hypersurfaces}\label{sec: Hodge examples}
In this subsection, we compute the primitive component of the Chern character of structure sheaves on certain complete intersections in $\mathbb P^{2k+1}$ that are subvarieties of the Fermat hypersurface $X_m^{2k}$. This will recover and enhance the work of \cite{Ran, ShiodaHodge, AokiShioda, Aoki} and answer Question 1 of \cite{daSilva}. 

\begin{example}
The following is a complete intersection considered in \cite{Ran, ShiodaHodge}. Take the complete intersection $Z = Z(f_0, \dots, f_k) \subseteq X_m^{2k}$ given by $$f_i = x_{2i} - \zeta x_{2i+1},$$ where $\zeta$ is a primitive $2m$-th root of unity. This complete intersection corresponds to the Koszul factorization
$$
a_i = x_{2i} - \zeta x_{2i+1}, \qquad b_i = \sum_{\ell=0}^{m-1} x_{2i}^\ell(\zeta x_{2i+1})^{m-1-\ell}.
$$
Using Theorem~\ref{thm: koszul det chprim}, a direct computation from the definitions gives
$$
\op{ch}_k^{\op{prim}}(\O_Z) = \frac{-1}{mc_k}\det M_Z = \frac{(-1)^{k(k+1)/2+1}}{mc_k}((1-m)\zeta)^{k+1} \prod_{i=0}^k \left( \sum_{\ell_i = 0}^{m-2} x_{2i}^{\ell} (\zeta x_{2i+1})^{m-2-\ell}\right).
$$
By Propositions~\ref{prop: shift alpha and monomial} and ~\ref{prop: isolate algebraics fermat}, this implies that the classes in $V(\alpha)$ where
$$
\alpha = (\ell_0, m-\ell_0, \ell_1, m-\ell_1, \dots, \ell_n, m-\ell_n)
$$
for any $\ell_i \in \{1, \dots, m-1\}$ lies in the image of the complexified Chern character map, as well as any permutations of the $\alpha_i$. In \cite{Ran, ShiodaHodge}, these cycles were enough to prove the Hodge conjecture for Fermat hypersurfaces of prime degree.
\end{example}

\begin{example}
Suppose $m = 2d$. Take the curve on $X_{2m}^1$ given by the complete intersection $Z=Z(f_0, f_1) \subseteq \mathbb{P}^3$ where
\begin{align*}
    f_0 &:= x_0^d + x_1^d + ix_2^d;\\
    f_1 &:= x_3^2 - 2^{1/d} x_0x_1.
\end{align*}
Here, we can compute that this complete intersection corresponds to the Koszul factorization given by 
$$
a_0 = f_0; \quad a_1 = f_1; \quad b_0 = x_0^d + x_1^d - ix_2^d; \quad b_1 = \sum_{k=0}^{d-1} x_3^{2k} (2^{1/d} x_0x_1)^{d-1-k}.
$$
Using Theorem~\ref{thm: koszul det chprim} again, we can compute 
\begin{align*}
  \op{ch}_k^{\op{prim}}(\O_Z) =   \frac{-1}{mc_k}\det M_Z &= -4(2^{1/d}) d(d-1)i x_0^d x_{2}^{d-1} x_3 \sum_{k=1}^{d-1} x_3^{m-2-2k} (2^{1/d}x_0x_1)^{k-1}\\
  &\quad + 4(2^{1/d}) d(d-1)i x_1^d x_{2}^{d-1} x_3 \sum_{k=1}^{d-1} x_3^{m-2-2k} (2^{1/d}x_0x_1)^{k-1}.
\end{align*}
\end{example}
Again, by Propositions~\ref{prop: shift alpha and monomial} and ~\ref{prop: isolate algebraics fermat}, this implies that the classes in $V(\alpha)$ with 
$$
\alpha = (k, d+k, d, m-2k) \text{ for $k=1, \dots, d-1$}
$$
are in the span of the complexified algebraic classes, as well as any permutations of these. This recovers \cite[Theorem 1]{AokiShioda}.

\begin{example}\label{ex:daSilva} Consider the fourfold $X_m^4$ where 3 divides $m$. Write $\ell = \tfrac{m}{3}$. We have the following choice of $f_i$
$$
    f_0= \sum_{i} x_i^{\ell},\qquad 
    f_1 = \sum_{i<j} x_i^\ell x_j^{\ell}, \text{ and }, \qquad 
    f_2 = \sum_{i<j<k} x_i^\ell x_j^{\ell} x_k^{\ell}.
$$
In \cite[Question 1]{daSilva}, da Silva asks if the complete intersection $Z:=Z(f_0, f_1, f_2)$ could give new algebraic cycles for Fermat fourfolds $X_4^m$. As noted in loc. cit., the Fermat polynomial can be written as $a_0b_0+a_1b_1+a_2b_2$ where 
$$
    a_0 = f_0; \quad b_0 = f_0^2; \quad 
    a_1 = f_1; \quad b_1 = -3f_2; \quad 
    a_2 = f_2; \quad b_2 = 3.
$$
    Since $b_2$ is a constant, the matrix $M$ in Theorem~\ref{thm: koszul det chprim} has a column of zeros. Thus, by Theorem~\ref{thm: koszul det chprim}, we have $\op{ch}^{\op{prim}}_2(\O_Z) = 0$, answering da Silva's question in the negative. In \textsection\ref{sec: Hodge for deg 33}, we will give a way to give the Hodge class da Silva set out to find with this example and prove the Hodge conjecture for the degree 33 Fermat fourfold (\Cref{thm: Hodge conjecture for deg 33 Fermat}).
\end{example}

\subsection{Proof of the Hodge conjecture for the degree 33 Fermat fourfold}\label{sec: Hodge for deg 33}

Consider the Fermat polynomial 
$$
F_{33}^4 = x_0^{33} + x_1^{33} + x_2^{33} + x_3^{33} + x_4^{33} + x_5^{33}
$$
and corresponding fourfold $X_{33}^4 \subseteq\mathbb{P}^5$.
Using the methods of Aoki \cite{Aoki, Aoki2} and Shioda \cite{ShiodaHodge}, da Silva \cite{daSilva} proved that the Hodge conjecture for $X_{33}^4$ reduces to establishing that the eigenspace $V(\alpha)$ is contained in the   complexified image of the Chern character map, where
$$
\alpha = (19,7,13,10,28,22).
$$

This class has proven elusive in the past: it and the permutations of its Galois orbit are all classes which are not quasi-decomposable in Shioda's sense. Hence, previous techniques have not realized them as algebraic \cite{daSilva}. Nonetheless, we now prove the following.
\begin{theorem}\label{thm: Hodge conjecture for deg 33 Fermat}
The eigenspaces $V(\alpha)$ is in the complexified span of the algebraic classes. Hence, the Hodge conjecture is true for $X_{33}^4$. 
\end{theorem}

To prove this theorem, we will pull back classes from a special cubic fourfold to $X_{33}^4$. Consider the polynomial
\begin{equation}
F_A := x_0^2 x_1 + x_1^2 x_2 + x_2^2 x_3 + x_3^2 x_4 + x_4^2 x_0 + x_3^3.
\end{equation}
and its corresponding cubic fourfold $X_A = Z(F_A) \subset \mathbb{P}^5$. This fourfold is well-studied. As seen in
\cite{BGM}, it has a symplectic automorphism of order 11, is rational, $H(X_A, \Z)\cap H^{2,2}$ has rank 21, and it has 6270 families of cubic scrolls which generate the lattice $H(X_A, \Z)\cap H^{2,2}$. The Griffiths Residue Theorem (see \Cref{Griffiths Residue Theorem}) implies that $H^{2,2}_{\op{prim}}(X_A)$ is 20 dimensional.  Indeed, we have the following explicit basis 
\begin{equation}\label{eq: basis Jac XA}
\op{Res}(\Omega^{x_i x_j x_k}), \text{ where } 0 \le i < j < k \le 5,
\end{equation}
for $H^{2,2}_{\op{prim}}(X_A)$  \cite[Theorem 2.10]{HLSW}.  Since the primitive Hodge lattice is also 20-dimensional, it follows that the complexified image of the Chern character map to $X_A$ also has the explicit image above.

Thus, we have that, for any $(i_0,j_0,k_0)$ with $0 \le i_0 < j_0 < k_0 \le 5$, there exists a cubic scroll $S(i_0,j_0,k_0)$ so that
$$
\op{ch}_2^{\op{prim}}(\O_{S(i_0,j_0,k_0)}) = \sum_{0 \le i'< j'<k'\le 5} c_{(i,j,k)} \op{Res}(\Omega^{x_ix_j x_k}) \in H^{2,2}_{\op{prim}}(X_A)
$$
and $c_{(i_0,j_0,k_0)} \ne 0$.

Define the rational map
\begin{equation}\label{eq: shioda map 1}
\phi: X_{33}^4 \dashrightarrow X_A, \qquad 
(x_0: \dots: x_5) \mapsto (y_0: \dots: y_5),
\end{equation}
where 
\begin{equation}\begin{aligned}\label{eq: shioda map}
    y_0 &= x_0^{16}x_1^{-8}x_2^4x_3^{-2}x_4,
    &y_1 &= x_0x_1^{16}x_2^{-8}x_3^4x_4^{-2},\\
    y_2 &= x_0^{-2}x_1x_2^{16}x_3^{-8}x_4^4,
    &y_3 &= x_0^4x_1^{-2}x_2x_3^{16}x_4^{-8},\\
    y_4 &= x_0^{-8}x_1^4x_2^{-2}x_3x_4^{16},
    &y_5 &= x_5^{11}.
\end{aligned}\end{equation}
\begin{remark} This map is an example of a Shioda map, first introduced by Shioda to study Delsarte surfaces and birational geometry of mirrors \cite{ShiodaMaps, BvGK, Bini, Kelly}. \end{remark}

\begin{proof}[Proof of \Cref{thm: Hodge conjecture for deg 33 Fermat}]
We pullback the forms in ~\eqref{eq: basis Jac XA} via the rational map $\phi$ to obtain that
$$
\phi^*\left(\op{Res}\Omega^{x_ix_jx_k}\right) = \op{Res}\left(1185921 \frac{y_iy_jy_k (x_0^{10}x_1^{10}x_2^{10}x_3^{10}x_4^{10}x_5^{10})\Omega}{(F_{33}^4)^3}\right),
$$
where $y_i, y_j, y_k$ are as in ~\eqref{eq: shioda map} (see \cite[\textsection 2.7]{BvGK}). Using the fact that 
\begin{align*}
y_0y_4y_5(x_0^{10}x_1^{10}x_2^{10}x_3^{10}x_4^{10}x_5^{10}) &= (x_0^{16}x_1^{-8}x_2^4x_3^{-2}x_4)(x_0^{-8}x_1^4x_2^{-2}x_3x_4^{16})( x_5^{11})(x_0^{10}x_1^{10}x_2^{10}x_3^{10}x_4^{10}x_5^{10})\\
&= x_0^{18}x_1^{6}x_2^{12}x_3^{9}x_4^{27}x_5^{21},
\end{align*}
we get, in particular, the equation
$$
\phi^*(\op{Res}(\Omega^{x_0x_4x_5})) = 1185921\op{Res}(\Omega^{ x_0^{18}x_1^{6}x_2^{12}x_3^{9}x_4^{27}x_5^{21}}) \in V(\alpha_1).
$$

Take the cubic scroll $S_{(0,4,5)}$ on $X_A$. We now pullback via $\phi$ its structure sheaf. Note that
$$
\op{ch}_2^{\op{prim}}(\phi^* \O_{S_{(0,4,5)}}) = \phi^*\op{ch}_2^{\op{prim}}(\O_{S_{(0,4,5)}})) =  \sum_{0 \le i< j<k\le 5} c_{(i,j,k)} \phi^*\op{Res}(\Omega^{x_ix_j x_k}).
$$
where $c_{(0,4,5)} \ne 0$, hence the coefficient of $\op{Res}(\Omega^{x_0^{18}x_1^{6}x_2^{12}x_3^{9}x_4^{27}x_5^{21}})$ in $\op{ch}_2^{\op{prim}}(\phi^* \O_{S_{(0,4,5)}})$ is nonzero. Thus $V(\alpha)$ is in the span of the complexified algebraic classes by \Cref{prop: isolate algebraics fermat}.
\end{proof}

\begin{remark}\label{rmk: no Koszul} It would be satisfying to find an ACM sheaf and its corresponding 2-periodic resolution that realizes these new algebraic classes. We conjecture that no Koszul factorization on $X_{33}^4$ will have a Chern character with a nonzero coefficient in front of the monomial $x_0^{18}x_1^{6}x_2^{12}x_3^{9}x_4^{27}x_5^{21}$. We remark there has been some work using AI looking to find such a Koszul factorization with such a monomial which has bore no fruit \cite{Chojecki}.
\end{remark}

\bibliography{FK}
\bibliographystyle{amsalpha}
\end{document}